\documentclass[11pt]{article}
\usepackage{stmaryrd}
\usepackage{graphicx}
\usepackage{color}
\usepackage[title]{appendix}
\usepackage{amsmath, amsfonts,amscd,amssymb, mathtools,mathrsfs, dsfont}

\usepackage{pgfplots}
\usepackage{float}

\usepackage{url}
\usepackage{hyperref}

\usepackage{array}

\usepackage{setspace}

\allowdisplaybreaks[4]

\newif\ifblog
\newif\iftex
\blogfalse
\textrue

\newcommand{\thmref}[1]{Theorem~{\rm \ref{#1}}}
\newcommand{\lemref}[1]{Lemma~{\rm \ref{#1}}}

\newcommand{\propref}[1]{Proposition~{\rm \ref{#1}}}
\newcommand{\defref}[1]{Definition~{\rm \ref{#1}}}

\newcommand{\be}{\begin{eqnarray}}
\newcommand{\ee}{\end{eqnarray}}
\newcommand{\bee}{\begin{eqnarray*}}
\newcommand{\eee}{\end{eqnarray*}}

\newcommand{\cc}{\overline{c}}
\newcommand{\uc}{\underline{c}}

\def\d{{\:\rm d}}
\def\dt{\:{\rm d}t}
\def\ds{\:{\rm d}s}
\def\dx{\:{\rm d}x}
\def\dy{\:{\rm d}y}
\def\dz{\:{\rm d}z}
\def\P{{\mathbb P}}
\def\E{{\mathbb E}}

\def\R{{\mathbb R}}
\def\N{{\mathbb N}}

\def\R{{\mathbb R}}

\def\D{{\mathbb D}}
\def\Q{{\cal Q}}
\def\SS{{\cal S}}

\def\NS{{\cal NS}}

\def\p{{\partial}}

\newcommand{\cF}{{\cal F}}
\newcommand{\ep}{\varepsilon}

\newcommand{\nd}{\noindent}

\newcommand{\la}{\lambda}

\newcommand{\ga}{\gamma}

\renewcommand{\geq}{\geqslant}
\renewcommand{\leq}{\leqslant}
\renewcommand{\ge}{\geqslant}
\renewcommand{\le}{\leqslant}

\newcommand{\esssup}{\mathop{\rm ess\:sup}\limits}

\def\sw{\mathcal{X}}

\def\swb{\mathfrak{M}}

\newtheorem{theorem}{Theorem}[section]
\newtheorem{lemma}[theorem]{Lemma}
\newtheorem{definition}[theorem]{Definition}

\newtheorem{proposition}[theorem]{Proposition}

\newenvironment{proof}{\noindent{\sc Proof:}}{\strut\hfill $\Box$\medskip\\} 
\newenvironment{proofof}[1]{\noindent\hspace{0ex}\textbf{Proof of #1}\hspace{1ex}}{\hspace*{\fill}$\Box$\newline}

\newcommand{\LL}{{\cal L}}
\newcommand{\TT}{{\cal T}}
\newcommand{\JJ}{{\cal J}}
\newcommand{\KK}{{\cal K}}
\newcommand{\II}{{\cal I}}

\def\ov{\overline{v}}

\def\os{\overline{s}}

\def\Ds{\Delta s}
\def\Dc{\Delta c}
\def\A{{\cal A}}

\def\sw{\mathcal{X}}

\newcommand{\nn}{\nonumber}

\begin{document}
\title
{Dividend ratcheting and capital injection under the Cram\'er-Lundberg model: Strong solution and optimal strategy}

\author{Chonghu Guan\thanks{School of Mathematics, Jiaying University, Meizhou 514015, Guangdong, China. Email: \url{gchonghu@163.com} }
\and Zuo Quan Xu\thanks{Department of Applied Mathematics, The Hong Kong Polytechnic University, Kowloon, Hong Kong SAR, China. Email: \url{maxu@polyu.edu.hk} }
}
\date{}

\maketitle 
\begin{abstract}
We consider an optimal dividend payout problem for an insurance company whose surplus follows the classical Cram\'er-Lundberg model. The dividend rate is subject to a ratcheting constraint (i.e., it must be nondecreasing over time), and the company may inject capital at a proportional cost to avoid ruin. This problem gives rise to a stochastic control problem with a self-path-dependent control constraint, costly capital injections, and jump-diffusion dynamics. The associated Hamilton-Jacobi-Bellman (HJB) equation is a partial integro-differential variational inequality featuring both a nonlocal integral term and a gradient constraint.

We develop a systematic probabilistic and PDE-based approach to solve this HJB equation. By discretizing the space of admissible dividend rates, we construct a sequence of approximating regime-switching systems of ordinary integro-differential equations. Through careful a priori estimates and a limiting argument, we prove the existence and uniqueness of a \emph{strong solution} in a suitable space. This regularity result is fundamental: it allows us to characterize the optimal dividend policy via a switching free boundary and to construct an explicit optimal feedback control strategy. To the best of our knowledge, this is the first complete solution --- comprising both the value function and an implementable optimal strategy --- for a dividend ratcheting problem with capital injection under the Cram\'er-Lundberg model. Our work advances the mathematical theory of optimal stochastic control beyond the standard viscosity solution framework, providing a rigorous foundation for dividend policy design in economics.

\bigskip
\nd {\bf Keywords.} Dividend ratcheting, capital injection, Cram\'er-Lundberg model, free boundary problem, variational inequality, integro-differential equation.

\bigskip
\nd {\bf 2010 Mathematics Subject Classification.} 35R35; 35Q93; 91G10; 91G30; 93E20; 45K05.

\end{abstract}

\tableofcontents

\section{Introduction}
The optimal dividend payout problem has been a cornerstone of actuarial science and financial mathematics since the seminal works of De Finetti \cite{DeFinetti1957} and Gerber \cite{Gerber1969}. At its core, the problem seeks a dynamic strategy for distributing surplus to shareholders that maximizes the expected present value of future dividends, balancing the immediate reward against the risk of premature bankruptcy. The literature has extensively explored various facets of this problem, including different surplus processes (e.g., compound Poisson \cite{GerberShiu2006}, Brownian motion \cite{AsmussenTaksar1997,Ta00,GerberShiu2004}), control types (e.g., impulse, barrier), and operational constraints.

A more recent and practically relevant line of inquiry incorporates path-dependent constraints on the dividend policy itself. Notably, the \emph{ratcheting constraint} requires the dividend payout rate to be non-decreasing over time, reflecting the managerial reluctance or contractual inability to cut dividends. This problem, under Brownian motion setting, was tackled by Albrecher, Azcue, and Muler \cite{Albrecher2020,Albrecher2022} using viscosity solution theory, and later resolved completely by Guan and Xu \cite{GuanXu2024} using a novel PDE method that yielded a strong solution and an explicit optimal feedback strategy. This PDE approach was further extended to incorporate a more flexible \emph{drawdown constraint} by Guan, Fan, and Xu \cite{GuanFanXu}, where the dividend rate is allowed to decrease but not below a fixed proportion of its historical maximum. That work resulted in a strong solution for a nonlinear HJB variational inequality.

In this paper, we build upon this established PDE framework to study a more comprehensive and realistic model. We consider an insurance company whose surplus process is governed by the classical \emph{Cram\'er-Lundberg model}, which incorporates both a deterministic income stream and random losses from a compound Poisson process. This is a more accurate representation of insurance risk than the pure diffusion models used in our previous works. Furthermore, to ensure the company's solvency and prevent ruin, we allow for costly \emph{capital injections}. This feature allows the surplus to be replenished from external sources (e.g., issuing new equity) at a premium cost, making the model more applicable in practice.

The literature on optimal dividend problems with capital injection is extensive. Early works by Sethi and Taksar \cite{SethiTaksar2002} and Sethi, Taksar, and Yan \cite{Sethi2005} studied dividend and capital injection problems in a diffusion setting, establishing that a barrier strategy is optimal. Kulenko and Schmidli \cite{KulenkoSchmidli2008} extended this to the Cram\'er-Lundberg model, showing that the optimal policy involves paying dividends at a barrier and injecting capital to keep the surplus nonnegative. Later, He, Liang, and Yuan \cite{HeLiangYuan2015} considered a problem with both dividend and capital injection under a spectrally negative Lévy process. More recently, Keppo, Reppen, and Soner \cite{Keppo2021} examined discrete dividend payments with capital injections, highlighting the trade-off between distribution and solvency. In all these works, the dividend policy is typically a barrier strategy without path-dependent constraints. Our work differs fundamentally by imposing a ratcheting constraint on the dividend rate, which introduces a self-path-dependent element that, when combined with capital injections and a jump-diffusion surplus, leads to a substantially more complex stochastic control problem.

The integration of these three features --- a jump surplus process, costly capital injections, and a ratcheting dividend constraint --- leads to a significantly more complex optimal control problem. The associated Hamilton-Jacobi-Bellman (HJB) equation is a new type of variational inequality that combines a partial integro-differential operator with two gradient constraints. This is markedly different from the linear differential operator in \cite{GuanXu2024} and the nonlinear differential operator in \cite{GuanFanXu} under Brownian motion setting. The non-local nature of the integral term, caused by the presence of compound Poisson process, introduces substantial analytical challenges, particularly in establishing comparison principles and proving the regularity of the solution.

Our primary contribution is to show that, despite these difficulties, the HJB equation admits a unique \emph{strong solution}. This represents a crucial improvement over the viscosity solution approach, which --- while powerful for existence --- does not suffice for constructing an implementable optimal feedback strategy. The strong regularity of our solution enables us to:

\begin{itemize}
\item Provide a complete characterization of the optimal strategy;
\item Show that the optimal dividend policy is governed by a free boundary separating the state space into regions where the dividend rate should be increased versus kept constant;
\item Explicitly incorporate the gradient constraint arising from costly capital injections; and
\item Construct a feedback control, via the explicitly determined switching free boundary, that is both admissible and optimal.
\end{itemize}

Our method follows a similar spirit to our previous works \cite{GuanXu2024,GuanFanXu}. We first analyze a boundary case to determine the terminal condition for the HJB equation. We then discretize the problem by considering a finite set of possible dividend rates, leading to a regime-switching system of ordinary integro-differential equations (OIDEs). By establishing uniform estimates for this approximating system and passing to the limit, we construct a strong solution to the original HJB variational inequality. This approach, which is more systematic than the guess-and-verify methods prevalent in stochastic control, yields a wealth of qualitative properties of the value function, including its concavity and boundedness of its derivatives. Ultimately, these properties allow us to prove the existence and continuity of the optimal switching boundary and to verify that the candidate strategy is indeed optimal.

The remainder of this paper is organized as follows. Section~\ref{sec:pf} formulates the optimal dividend ratcheting problem with capital injection under the Cramér--Lundberg model. Section~\ref{sec:hjb} derives the associated HJB variational inequality and presents the optimal strategy via a free boundary. Section~\ref{sec:solution} presents the core technical contribution of the paper: the construction of a strong solution to the HJB equation via a regime-switching approximation and a limiting argument. Section~\ref{sec:FB} establishes the regularity properties of the free boundary and the equivalent maximum rate function. Finally, Appendix~\ref{sec:proofs} contains the technical proofs and auxiliary results. 

\paragraph*{Notation.}
We use $\R$ to denote the set of real numbers, and $\R^+$ the set of non-negative real numbers.
For any measurable set $\D\subseteq\R$ and positive integer $n$,
define the space
$$
W^{n,\infty}(\D)=\Big\{u:\D\mapsto \R \;\Big|\; \esssup\limits_{x\in \D} \Big|\frac{\d^k u}{\dx^k}\Big| <+\infty,\; k=0,1,\cdots,n \Big\},
$$
where $\frac{\d^k u}{\dx^k}$ is the $k$-order weak derivative of $u$, and define the space
$$
C^n(\D)=\Big\{u\in W^{n,\infty}(\D) \;\Big|\; \frac{\d^k u}{\dx^k}\in C(\D),\; k=0,1,\cdots,n \Big\},
$$
where $C^0(\D)=C(\D)$ is the set of continuous functions on $\D$.

\section{Problem formulation and preliminaries} \label{sec:pf}
\setcounter{equation}{0}

Our model is established in a filtered complete probability space $(\Omega, \cF, \P, \{\cF_t\}_{t\geq0})$ satisfying the usual conditions.
The surplus $X_t$ of an insurance company is an $\{\cF_t\}_{t\geq0}$-adapted process satisfying the Cram\'er-Lundberg model with controls of dividend payout and capital injection:
\begin{align}\label{X_eq}
X_t=x+\int_0^t (\mu-C_s) \ds-\sum_{i=1}^{N_t} Z_i+D_t,\quad t\geq 0.
\end{align}
Here, $x$ is the initial surplus, $\mu>0$ is the constant income rate, $C_t$ is the dividend payout rate at time $t$, $D_t$ is the accumulated capital injection until time $t$,
$\{N_t\}_{t\geq 0}$ is a Poisson process with a constant intensity $\la> 0$,
$\{Z_i\}_{i=1}^\infty$ is a a series of independent random variables with a common distribution function $F$,
all of which are independent of $\{N_t\}_{t\geq 0}$.

We use $\Pi_{x,c}$ to denote the set of admissible dividend-capital injection strategies $\{(C_t,D_t)\}_{t\geq0}$ that satisfy the following properties. 
\begin{enumerate}
\item First, both the dividend payout process $\{C_t\}_{t\geq0}$ and the cumulated capital injection process $\{D_t\}_{t\geq0}$ shall be $\{\cF_t\}_{t\geq0}$-adapted and c\`{a}dl\`{a}g (right-continuous with left limits).

\item
Second, the dividend rate process $\{C_t\}_{t\geq 0}$ satisfies $c \leq C_t \leq \cc$ for all $t \geq 0$ and obeys the \textit{ratcheting constraint}: it is non-decreasing. The constant $\cc \in (0,\mu)$ serves as an upper bound on the dividend payout rate.
Clearly, shareholders expect to receive positive dividend payouts, at least in the maximal payout scenario; it is therefore economically natural to assume $\cc > 0$. 
The upper bound $\cc < \mu$ is equally natural from an economic perspective. Indeed, if $\cc \geq \mu$, the surplus would lack a positive drift --- even at the maximal dividend rate --- to offset incoming claims. This would force frequent, costly capital injections and ultimately diminish the company's performance. Hence, the condition $\cc \in (0,\mu)$ is both economically reasonable and necessary for a nontrivial dividend optimization problem.
\item 
Third, the cumulated capital injection process $\{D_t\}_{t\geq0}$ shall be nonnegative and non-decreasing; for ease of presentation, we take the convention that $D_t=0$ for all $t<0$.
\item
Last, the surplus process $X$ in \eqref{X_eq} under the strategy $\{(C_t,D_t)\}_{t\geq0}$ shall not go bankrupt, i.e., it holds that $X_t\geq 0$ for all $t\geq 0$.
\end{enumerate} 
For any $(x,c)\in \mathbb{R}\times(-\infty,\cc]$, one can verify that ${(\overline{c}, \sum_{i=1}^{N_t} Z_i + |x|)}_{t\geq 0}$ is an admissible strategy; hence $\Pi_{x,c}$ is nonempty.
Note that we allow the initial surplus $x$ to be negative, as capital injection can immediately render the surplus positive. We also permit $c$ to take negative values, which may be interpreted as the company issuing bonds at a positive return rate, albeit with an adverse effect on performance. This flexibility, together with the ratcheting constraint, broadens the model's scope and enhances its alignment with real-world financial practices.

The company's objective is to find an admissible dividend-capital injection strategy $\{(C_t,D_t)\}_{t\geq0}\in \Pi_{x,c}$ to maximize the expectation of future discounted cumulative dividend payouts after reducing the cost of capital injection:\footnote{Here and hereafter, we adopt the notation that
\[\int_0^\infty e^{-r t} \d D_t:=\int_{[0,\infty)} e^{-r t} \d D_t, \]
which is also equal to
\[\int_{(0,\infty)} e^{-r t} \d D_t+D_0
=\int_{(0,\infty)} e^{-r t} \d D^c_t+\sum_{t\geq 0}e^{-r t}\Delta D_t,\]
where $D^c$ is the continuous part of the process $D$ and $\Delta D_t=D_t-D_{t-}$.}
\begin{align}\label{value}
V(x,c)=
\sup\limits_{\{(C_t,D_t)\}_{t\geq0}\in \Pi_{x,c}}\E \Bigg[\int_0^\infty e^{-r t} C_t \dt-\ell \int_0^\infty e^{-r t} \d D_t \Bigg],~~~(x,c)\in \R\times (-\infty,\cc],\bigskip
\end{align}
where $r>0$ represents a constant discount factor and $\ell$ represents the cost per unit of capital injection.

We assume $\ell > 1$ throughout this paper. 
This condition is economically natural: raising external capital entails transaction costs, underwriting fees, and adverse selection, making it more expensive than retaining internal surplus. 
If $\ell \le 1$, external capital would be cheaper than or equal to retained earnings, leading to degenerate behavior (e.g., unbounded value functions or collapse of the free boundary). 
Mathematically, $\ell > 1$ ensures the gradient constraint $v_x \le \ell$ is non-trivial and guarantees the well-posedness of the HJB variational inequality.

We also take the following technical assumption throughout the paper:
\begin{align}\label{C1}
\d F(x)=p(x) \dx,~~\hbox{where $p$ is positive, bounded and non-increasing on $(0,\infty)$,}
\end{align}
and
\begin{align*}
0<\ga:=\E [Z_1]=\int_0^\infty xp(x)\dx <\infty.
\end{align*}
The boundedness and positivity assumptions on $p$ are not essential and are imposed only to avoid unnecessary technicalities, as the slight increase in generality would not justify the added complexity in the proofs.

The remainder of this paper is devoted to the study of problem \eqref{value}. We begin by establishing several fundamental properties of the value function. These properties, in turn, inspire and guide our subsequent analysis of the associated HJB equation.
\begin{lemma}\label{lem:lipchitz}
The value function $V$ defined in \eqref{value} is monotonically decreasing in $c$, concave in $x$ and satisfies
\begin{align*}
0\leq V(y,c)-V(x,c)\leq \ell (y-x), &~~\text{if}~x\leq y,~c\leq \cc;\\
V(x,c)=V(0,c)+\ell x, &~~\text{if}~x\leq 0,~c\leq \cc;\\
\frac{ \cc-\la\ell\ga}{r}\leq V(x,c) \leq \frac{ \cc}{r},&~~\text{if}~x\geq 0,~c\leq \cc.
\end{align*}
Moreover,
for any $x< 0$, a strategy $\{(C_t,D_t)\}_{t\geq0}$ is an optimal for $V(0,c)$ if and only if $\{(C_t,D_t-x)\}_{t\geq0}$ is optimal for $V(x,c)$; in particular, it suffices to find an optimal solution for $V(0,c)$ so as to solve $V(x,c)$.
\end{lemma}
Its proof is given in Appendix \ref{proof:lem:lipchitz}. Observe that this result yields the key estimate $0\leq V_x\leq \ell$. As will become evident, this bound is essential for our subsequent investigation of the HJB equation. 
Finally, we note that it suffices to study problem \eqref{value} in the region
\[\Q^{+}_{\infty}:=\R^+\times (-\infty, \cc],\]
since the behavior for $x< 0$ can be transformed to the case $x=0$ by \lemref{lem:lipchitz}. 


\section{Solution to the problem \eqref{value}}\label{sec:hjb}

The goals of this section are twofold. First, we introduce the HJB equation for problem \eqref{value} and establish the uniqueness of its solution in a strong sense. Second, assuming the existence of such a solution and the properties of an associated free boundary, we provide a complete characterization of the optimal strategy, thereby solving problem \eqref{value}. The existence proof and the detailed analysis of the free boundary, which are more involved, are deferred to Sections \ref{sec:solution} and \ref{sec:FB}, respectively.

We will apply PDE method to study the problem \eqref{value}. To this end, we need first to establish a comparison principle. It will be critical for the subsequent analysis, and many results of this paper will relay on it.

\subsection{A comparison principle}

Throughout the paper, we use the following three linear operators on functions $v(x,c):\R\times (-\infty, \cc]\to \R$:
\begin{align*}
\LL_c v:=&~ -(\mu-c) v_x+(r+\la) v, \medskip\\
\TT v:=&~\la \int_0^x v(x-y,c)\d F(y)+\la v(0,c)(1-F(x))=\la\E[v((x-Z_1)^+,c)], \medskip\\
\II v:=&~ \la \int_0^x v(x-y,c) \d F(y).
\end{align*}
Notice we have $\p_x (\TT v)=\II v_x$.
For any function $f(x):\R\to\R$, we shall interpret $f(x)$ as $f(x,\cc)$ when applying the above operators.

\begin{lemma}[Comparison principle]\label{lem:CP1}
Let $\D$ be a measurable subset of $\R^+$, $c\leq\cc$, $\eta\in W^{1,\infty}(\R^+)$, $\JJ: C(\R^+)\mapsto C(\R^+)$ be a linear operator satisfying
\begin{align}\label{J}
\JJ f(x)\leq \la \max\limits_{y\in [0,x]} f^+(y),\quad ~ x\in \R^+,
\end{align}
and $H(x,y):\R^+\times \R^+\to\R$ be a non-decreasing function in $y$.

Suppose $\psi_1,\;\psi_2\in W^{1,\infty}(\R^+)$ satisfy
\begin{align*}
\begin{cases}
\LL_c \psi_1-\JJ \psi_1+H(x,\psi_1)\leq\LL_c \psi_2-\JJ \psi_2+H(x,\psi_2) {\rm\quad a.e.\;in}~~ \D,\medskip\\
\psi_1\leq \psi_2{\rm\quad in}~~ \R^+\setminus\D,
\end{cases}
\end{align*}
or
\begin{align*}
\min\big\{\LL_c \psi_1-\JJ \psi_1+H(x,\psi_1), \; \psi_1-\eta\big\}\leq
\min\big\{\LL_c \psi_2-\JJ \psi_2+H(x,\psi_2), \; \psi_2-\eta\big\}~~{\rm a.e.\;in}~~ \R^+,
\end{align*}
then $\psi_1\leq \psi_2$ in $\R^+$.
\end{lemma}
Its proof is given in Appendix \ref{proof:lem:CP1}.

Notice
\begin{align*}
\TT f(x) &=\la \int_0^x f(x-y)\d F(y)+\la f(0)(1-F(x))\\
&\leq \la \max\limits_{y\in [0,x]} f^+(y)\cdot\Big(\int_0^x 1\d F(y)+(1-F(x))\Big)=\la \max\limits_{y\in [0,x]} f^+(y)
\end{align*}
and
\begin{align*}
\II f(x) \leq \la \max\limits_{y\in [0,x]} f(y) \cdot
\int_0^x 1\d F(y)\leq \la \max\limits_{y\in [0,x]} f^+(y),\quad x\in \R^+,
\end{align*}
so we can apply \lemref{lem:CP1} with $\JJ=\II$ and $\JJ=\TT$.

\subsection{Boundary problem: $V(x,\cc)$}
Different from classical control problems, the boundary value of the problem \eqref{value} is not immediately known. To introduce the HJB equation, we need first to solve the boundary problem: $V(x,\cc)$.

Define
\begin{align}\label{defh}
h(x):=\la\ell\E[(Z_1-x)^+]=
\la\ell \int_x^\infty (1-F(y))\dy
=\la\ell \int_x^\infty (y-x)p(y) \dy,~~x\in\R^+.
\end{align}

\begin{lemma}\label{lem:g}
The following OIDE on $g$:
\begin{align}\label{g_eq}
\LL_{\cc} g-\TT g+h-\cc=0, \quad x\in\R^+,
\end{align}
has a unique bounded solution $g\in C^1(\R^+)$ which satisfies
\begin{gather}
\frac{\cc-\la\ell\ga}{r}\leq g\leq \frac{\cc}{r},\label{g_b}\\
0\leq g' \leq \ell, \label{g'}\\
g'' \leq 0, \label{g''}
\end{gather}
in $\R^+$ and
\begin{align}\label{g4}
g(+\infty):=&~\lim\limits_{x\to+\infty} g(x)=\frac{\cc}{r},\\
g'(+\infty):=&~\lim\limits_{x\to+\infty} g'(x)=0. \label{g5}
\end{align}
\end{lemma}
Its proof is given in Appendix \ref{proof:lem:g}.

From now on we let $g$ be the solution given in the above result.
It is indeed the optimal value for the problem \eqref{value} when $c=\cc$.

\begin{theorem}[Optimal strategy in the boundary case]\label{theorem:boundaryvalue}
We have $V(x,\cc)=g(x)$ for $x\in\R^+$ and $V(x,\cc)=g(0)+\ell x$ for $x\leq 0$, where $g$ is given in \lemref{lem:g}.
Given $x\in\R$, then $\{(\cc,\overline{D}_t)\}_{t\geq0}$ is an optimal strategy for $V(x,\cc)$, where
\begin{align*}
\overline{D}_t=\sup\limits_{0\leq s\leq t}\Big(\sum_{i=1}^{N_s} Z_i-x-(\mu-\cc)s\Big)^+,~~t\geq 0.
\end{align*}
\end{theorem}
Its proof is given in Appendix \ref{proof:theorem:boundaryvalue}.

\subsection{HJB equation on $\Q^{+}_{\infty}$}\label{sec:HJB}
The HJB equation for the problem \eqref{value} on $\Q^{+}_{\infty}$ is a variational inequality (VI) on $v(x,c)$:
\begin{align}\label{v_pb00}
\begin{cases}
\min\{\LL_c v-\TT v+h-c, \;\ell-v_x, \;-v_c\}=0, & (x,c)\in\Q^{+}_{\infty},\medskip\\
v(x,\cc)=g (x), & x\in\R^+,
\end{cases}
\end{align}
with bounded growth condition.

We will show \eqref{v_pb00} admits a {\it strong solution}, stronger than viscosity solution (see \cite{YZ99}).
This allows us to give a complete answer to the problem \eqref{value} later.
In words, a strong solution to \eqref{v_pb00} must: (i) be continuously differentiable in $x$; (ii) be non-increasing in
$c$; (iii) satisfy the variational inequality point-wisely; (iv) have the gradient constraint active only when the dividend rate is not at its maximum. The precise definition is given as follows.

\begin{definition}[Strong solution to \eqref{v_pb00}]\label{def:solution}
We call $v$ is a strong solution to the VI \eqref{v_pb00} if the follows hold:
\begin{enumerate}
\item $v\in\A_{\infty}:=\Big\{u:\Q^{+}_{\infty}\mapsto \R \;\Big|\; u \in C(\Q^{+}_{\infty})
\;\hbox{and}\; u(\cdot,c) \in C^1(\R^+) \;\hbox{for every } c\in (-\infty,\cc]\Big\}$;
\item $v$ is non-increasing w.r.t. $c$;
\item For each $c\in(-\infty,\cc]$,
\begin{align*}
\LL_c v-\TT v+h-c\geq 0 \;\hbox{ and }\;v_x\leq \ell\;\hbox{ in }\; \R^+;
\end{align*}
\item If for some $(x,c)\in \R^{+}\times (-\infty,\cc)$, we have $v(x,c)>v(x,s)$ for all $c<s\leq \cc,$ then
\begin{align*}
\big(\LL_{c} v - \TT v + h - c \big)(x,c)=0 \;\hbox{ or }\; v_x(x,c)=\ell;
\end{align*}
\item For all $x\in\R^+$, $v(x,\cc)=g(x)$.
\end{enumerate}
\end{definition}
\begin{lemma}[Uniqueness of strong solution to \eqref{v_pb00}]\label{lem:CP00}
The bounded strong solution to the VI \eqref{v_pb00}, if it exists, is unique.
\end{lemma}
Its proof is given in Appendix \ref{proof:lem:CP00}.

The main results of this paper is to show that the VI \eqref{v_pb00} admits a strong solution with nice properties. This solution will be shown to be the value function of the problem \eqref{value}.

\begin{theorem}[Existence of strong solution to \eqref{v_pb00}]\label{thm:u}
There exists a unique bounded strong solution $v$ to the VI \eqref{v_pb00}. The solution satisfies
\begin{align*}
\frac{\cc-\la\ell\ga}{r}\leq v \leq \frac{\cc}{r}, ~~
0\leq v_x\leq \ell, ~~
-\frac{\la\ell}{\mu-\cc}\leq v_{xx}\leq 0~{\rm a.e.,} ~~
-\frac{\ell-1}{r}\leq v_c\leq 0~{\rm a.e..}
\end{align*}
Furthermore, it is also the unique bounded strong solution (in the sense of \defref{def:solution2}) to the problem
\begin{align*}
\begin{cases}
\min\{\LL_c v-\TT v+h-c, \;-v_c\}=0, & (x,c)\in\Q^{+}_{\infty},\medskip\\
v(x,\cc)=g (x), & x\in\R^+,
\end{cases}
\end{align*}
i.e., if $v(x,c)>v(x,s)$ holds for all $s\in(c,\cc]$ at a point $(x,c)\in\Q^{+}_{\infty}$, then
\begin{align}\label{NS equation}
\big(\LL_{c} v - \TT v + h - c \big)(x,c)=0.
\end{align}
\end{theorem}
We will construct such a solution through a regime switching system.
Since the proof is very delicate, we defer it to Section \ref{sec:solution}.

\subsection{Optimal strategy to \eqref{value}.} \label{sec:optimalstrategy}
To find an optimal strategy for the problem \eqref{value},
we divided the domain $\R^+\times (-\infty, \cc)$ into a {\bf switching region}:
\begin{align*}
\SS:=\{(x,c)\in \R^+\times (-\infty, \cc) \mid v(x,c)=v(x,s)~\mbox{for some}~ c<s\leq \cc\}
\end{align*}
and a {\bf non-switching region}:
\begin{align*}
\NS:=\{(x,c)\in \R^+\times (-\infty, \cc) \mid v(x,c)>v(x,s)~\mbox{for all}~ c<s\leq \cc\}.
\end{align*}
By \eqref{NS equation}, we have $\LL_c v-\TT v+h-c=0$ in $\NS$.

The following shows they are separated by the following curve (free boundary)
\begin{align*}
\sw(c):=\inf\{x\in\R^+\mid (x,c)\in \SS\},\quad c\in(-\infty, \cc).
\end{align*}

\begin{proposition}[Property of the free boundary $\sw(\cdot)$]\label{proposition:freeboundary}
The limit $\sw(\cc):=\lim_{c\to\cc-}\sw(c)$ exists and finite. The curve $\sw(\cdot)$ is continuous in $(-\infty, \cc]$. The regions $\SS$ and $\NS$ are separated by it in the following sense:
\begin{align}\label{propertysw0}
\{(x,c)\in \R^+\times(-\infty, \cc)\mid x>\sw(c)\} \subseteq \SS \subseteq \{(x,c)\in \R^+\times(-\infty, \cc)\mid x \geq \sw(c)\},
\end{align}
and
\begin{align}\label{propertysw1}
\{(x,c)\in \R^+\times(-\infty, \cc)\mid x< \sw(c)\} \subseteq \NS \subseteq \{(x,c)\in \R^+\times(-\infty, \cc)\mid x \leq \sw(c)\}.
\end{align}
\end{proposition}
We will establish this result in Section \ref{sec:FB}.

To solve the problem \eqref{value}, we also define
the {\bf equivalent maximum rate} as
\begin{align}\label{swb}
\swb(x,c):=\max\Big\{c'\in [c,\cc] \;\Big|\; v(x,c')=v(x,c)\Big\},~~(x,c)\in\Q^{+}_{\infty}.
\end{align}
The following property is needed.
\begin{lemma}[Property of the equivalent maximum rate $\swb(\cdot)$]\label{lemma:swb}
For every $c< \cc$, the map $x\to\swb(x,c)$ is non-decreasing and right-continuous on $\R^+$. Moreover, \begin{align}\label{v_c=0}
v_c(x,\swb(x,c))=0~~{\rm if}\;c<\swb(x,c)<\cc,~x\in\R^+.
\end{align}
\end{lemma}
Its proof is given in Section \ref{proof:lemma:swb}.

Together with \lemref{lem:lipchitz}, we now provide a complete answer to the problem \eqref{value}.
\begin{theorem}[Optimal strategy in $\Q^{+}_{\infty}$]\label{theorem:strategy}
Let $v$ be the unique bounded strong solution to \eqref{v_pb00} given in \thmref{thm:u}. Then $v$ coincides with the optimal value to the problem \eqref{value} on $\Q^{+}_{\infty}$.
Moreover, given $(x,c)\in \Q^{+}_{\infty}$, $\{(C^*_t, D^*_t)\}_{t\geq 0}$ is an optimal strategy for the problem \eqref{value}, where the triple $\{(X^*_t, C^*_t, D^*_t)\}_{t\geq 0}$ is defined by
\[\begin{cases}
\displaystyle{X^*_t=x+\int_0^t(\mu-C^*_s) \ds -\sum_{i=1}^{N_t} Z_i + D^*_t,}\\[5mm]
C^*_t= \swb\Big(\max\limits_{s\in[0,t]}X^*_s,c\Big),\\[5mm]
\displaystyle{D^*_t=\sup\limits_{0\leq s\leq t}\Big(\sum_{i=1}^{N_s} Z_i-x-\int_0^s(\mu-C^*_\theta) \d \theta \Big)^+}.
\end{cases}
\]
\end{theorem}

{
The optimal strategy has an intuitive interpretation: The dividend rate $C^*_t$ is increased only when the running maximum surplus $\max\limits_{s\in[0,t]}X^*_s$ reaches a new high, and it is raised precisely to the level $\swb\Big(\max\limits_{s\in[0,t]}X^*_s,c\Big)$. Since the transaction cost $\ell>1$, one shall always use the minimum effect to avoid bankruptcy.
Thus, capital injections
$\Delta D^*_t$ occur only when a claim would otherwise drive surplus negative, and they are injected just enough to keep surplus nonnegative. This reflects a ``barrier-type'' policy modulated by the ratcheting constraint.
\bigskip\\
}
\begin{proofof}{\thmref{theorem:strategy}}
The proof is divided into four steps. Since several arguments are analogous to those in the proof of Theorem \ref{theorem:boundaryvalue}, we omit them for brevity. 
\paragraph{Step 1 (Upper bound: $v\geq V$).}
Given any $(x,c)\in \Q^{+}_{\infty}$ and any strategy $\{(C_t,D_t)\}_{t\geq0}\in \Pi_{x,c}$,
we get
\begin{align}\nonumber
v(x,c)
=&~
\E \Bigg[e^{-c T}v(X_{T},C_T)\Bigg]-\E \Bigg[\sum_{0\leq t\leq T}e^{-r t}\Big[v (X_t,C_t)-v\Big(X_t-(\Delta D_t-\Delta' D_t), C_{t-}\Big)\Big]\Bigg]\\\nonumber
&+\E \Bigg[\int_0^{T}e^{-r t}\big(\LL_{c} v - \TT v + h - c \big)(X_{t-},C_{t-}) \dt\Bigg]
-\E \Bigg[\int_0^{T}e^{-r t} v_c(X_{t-},C_{t-})\d C^c_t\Bigg]\\\label{vE}
&-\E \Bigg[\int_0^{T}e^{-r t} v_x(X_{t-},C_{t-})\d D^c_t\Bigg]
- \ell \E \Bigg[\sum_{0\leq t\leq T}e^{-r t}\Delta' D_t \Bigg]
+ \E\Bigg[\int_0^T e^{-r t} C_{t-} \dt\Bigg],
\end{align}
where
$$ \Delta' D_t:=(Z_{N_t} \Delta N_t-X_{t-})^+=(Z_{N_t} -X_{t-})^+ \Delta N_t.$$
Using \eqref{v_pb00}, $\d C^c_t \geq 0$ and $\Delta D_t\geq \Delta' D_t$, we obtain
\begin{align}\label{v>=}
v(x,c)
\geq
\E \Bigg[e^{-c T}v(X_{T},C_T)\Bigg]
+ \E\Bigg[\int_0^T e^{-r t} C_{t-} \dt\Bigg]
- \E\Bigg[\int_0^T e^{-r t} \ell \d D_t\Bigg].
\end{align}
Now sending $T \to +\infty$ yields
\begin{align*}
v(x,c)
\geq \E\Bigg[\int_0^\infty e^{-r t} C_{t} \dt-\int_0^\infty e^{-r t} \ell \d D_t\Bigg].
\end{align*}
Since $\{{C_t,D_t}\}_{t\geq 0}\in \Pi_{x,c}$ is arbitrary, we conclude $v(x,c)\geq V(x,c)$.

\paragraph{Step 2 (Admissibility).} 
We show that $\{(C_t^*, D_t^*)\} \in \Pi_{x,c}$.
\begin{itemize}
\item By \lemref{lemma:swb}, the mapping $y \mapsto \swb(y,c)$ is non-decreasing and right-continuous. Since $Y_t^* := \max_{0\le s\le t} X_s^*$ is non-decreasing and right-continuous, $C_t^* = \mathfrak{M}(Y_t^*, c)$ inherits these properties and satisfies $C_t^* \in [c,\overline{c}]$.

\item The process $D_t^*$ is defined as the minimal reflection that keeps $X_t^* \ge 0$; it is non-decreasing, right-continuous.
\end{itemize}
Hence $\{(C_t^*, D_t^*)\} \in \Pi_{x,c}$.

\paragraph{Step 3 (Lower bound and optimality: $v\leq V$).}
We now prove 
that inequality \eqref{v>=} is an equality when $\{{C_t,D_t}\}_{t\geq 0}=\{{C^*_t,D^*_t}\}_{t\geq 0}$, which will complete the proof of the theorem.

Recall $Y^*_t=\max\limits_{s\in[0,t]}X^*_s.$
We first prove
\begin{align}\label{LV*}
\big(\LL_{c} v - \TT v + h - c \big)(X^*_{t},C^*_{t})=0.
\end{align}
If $C^*_t<\cc$, then by the definition of $C^*_{t}$, we have $(Y^*_t, C^*_{t})\in \NS$. Since $X^*_t\leq Y^*_t$, it follows from \eqref{propertysw1} that $(X^*_t,C^*_t)\in \NS$, so \eqref{LV*} holds by \eqref{NS equation}.
If $C^*_t= \cc$, then $v(\cdot,C^*_{t})=g(\cdot)$ on $\R^+$, and \eqref{LV*} still holds by virtue of \eqref{g_eq}.

Second, it is easy to verify that
\begin{align}\label{DD*}
\Delta D^*_t = \Delta' D^*_t,~{ {D^*_t}^c=0,}~~ t \geq 0.
\end{align}

Third, we prove
\begin{align}\label{vcdct}
-v_c(X^*_{t-},C^*_{t-})\d C^{*c}_t=0,~~ t\neq \tau_1,\;\tau_2,
\end{align}
where $\tau_1=\inf\{t\geq 0\mid C^*_t > c\}$ and $\tau_2=\inf\{t\geq 0\mid C^*_t = \cc\}.$
For $t<\tau_1$ or $t>\tau_2$, we have $C^*_t = c$ or $C^*_t = \cc$, respectively, which implies $\d C^*_t =0$, so \eqref{vcdct} holds trivially.
We now assume $\tau_1< t <\tau_2$, so that $c < C^*_t < \cc$.
Note that \eqref{DD*} implies $\Delta X^*_t = \Delta' D^*_t -Z_{N_t}\Delta N_t\leq 0$.
If $X^*_t<Y^*_t$, then there exists a random time $\delta>0$ such that $X^*_s<Y^*_t$ for $s\in [t,t+\delta)$. It then follows $Y^*_s=Y^*_t$ and $C^*_s=C^*_t$ for $s\in [t,t+\delta)$, so \eqref{vcdct} holds.
If $X^*_t= Y^*_t$, then since $X^*_t \leq X^*_{t-} \leq Y^*_{t-} \leq Y^*_t$, we have $X^*_{t-}=X^*_t= Y^*_t$. Moreover, by the definition of $C^*_t$ and \eqref{v_c=0} we have $v_c(Y^*_t,C^*_t)=0$.
This together with $C^*_t \geq C^*_{t-}\geq c$ and $v(Y^*_t,s)=v(Y^*_t,c)$ for all $s\in [c,C^*_t)$ implies $v_c(X^*_{t-},C^*_{t-})=v_c(Y^*_{t},C^*_{t-})=0$, so \eqref{vcdct} also holds.

Fourth, we prove
\begin{align}\label{vv}
v(X^*_t,C^*_t)=v(X^*_t,C^*_{t-}).
\end{align}
This is trivial when $C^*_{t}=C^*_{t-}$.
Now suppose $C^*_{t}>C^*_{t-}$, then the preceding discussion shows that $X^*_t=Y^*_t$.
By the definition of $C^*_{t}$, we also have $v(Y^*_t,s)=v(Y^*_t,c)$ for all $s\in [c,C^*_t)$.
Combining above gives \eqref{vv} as $C^*_{t}>C^*_{t-}\geq c$.

Combining \eqref{LV*}, \eqref{DD*}, \eqref{vcdct}, and \eqref{vv}, we see 
\eqref{vE} becomes
$$
v(x,c)=
\E \Bigg[e^{-c T}v(X^*_{T},C^*_T)\Bigg]
+ \E\Bigg[\int_0^T e^{-r t} C^*_{t} \dt\Bigg]
- \E\Bigg[\int_0^T e^{-r t} \ell \d D^*_t\Bigg].
$$
By sending $T\to+\infty$, we obtain
$$
v(x,c)= \E\Bigg[\int_0^\infty e^{-r t} C^*_{t} \dt - \int_0^\infty e^{-r t} \ell \d D^*_t\Bigg].
$$
This establishes $v(x,c)\leq V(x,c)$.

\paragraph{Step 4 (Conclusion).}
Combining Steps 2 and 3 yields $
v(x,c)=V(x,c)$ and confirms the optimality of $\{{C^*_t,D^*_t}\}_{t\geq 0}$.
\end{proofof}

\section{Solvability of the HJB equation on $\Q^{+}_{\infty}$} \label{sec:solution}
This whole section is devoted to the
proof of \thmref{thm:u}.

To construct a solution to \eqref{v_pb00} with nice properties, we first introduce a sequence of regime switching problems which can be regarded as a sequence of option pricing problems. We then study its properties in subsection \ref{sec:approximation}.
In subsection \ref{sec:solvability}, we pass to the limit to get a solution to \eqref{v_pb00} and complete the proof of \thmref{thm:u}.

\subsection{Local HJB equation}
Since the domain $\Q^{+}_{\infty}$ is unbounded, it is not easy to study \eqref{v_pb00} directly.
We now introduce a simpler local version (called local HJB equation) of \eqref{v_pb00}:
\begin{align}\label{v_pb}
\begin{cases}
\min\{\LL_c v-\TT v+h-c, \;-v_c\}=0, & (x,c)\in\Q^{+}_{\uc}:=\R^+\times [\uc,\cc],\medskip\\
v(x,\cc)=g (x), & x\in\R^+.
\end{cases}
\end{align}
We emphasis that this VI looks simpler than \eqref{v_pb00} since the requirement $v_x\leq \ell$ is removed in above. Similar to \defref{def:solution}, we define the strong solution to \eqref{v_pb}.
\begin{definition}[Strong solution to \eqref{v_pb}]\label{def:solution2}
We call $v$ is a strong solution to the VI \eqref{v_pb} if the follows hold:
\begin{enumerate}
\item $v\in\A_{\uc}:=\Big\{u:\Q^{+}_{\uc}\mapsto \R \;\Big|\; u \in C(\Q^{+}_{\uc})
\;\hbox{and}\; u(\cdot,c) \in C^1(\R^+) \;\hbox{for every } c\in [\uc,\cc]\Big\}$;
\item $v$ is non-increasing w.r.t. $c$;
\item For each $c\in[\uc,\cc]$,
\begin{align}\label{-lv>=02}
\LL_c v-\TT v+h-c\geq 0 \;\hbox{ in }\; \R^+;
\end{align}
\item If $v(x,c)>v(x,s)$ holds for all $s\in(c,\cc]$ at a point $(x,c)\in\Q^{+}_{\uc}$, then
\begin{align}\label{-lv=02}
\big(\LL_{c} v - \TT v + h - c \big)(x,c)=0;
\end{align}
\item For all $x\in\R^+$, $v(x,\cc)=g(x)$.
\end{enumerate}
\end{definition}

The main result of this section is
\begin{proposition}\label{prop:u}
For any $\uc<\cc$, there exists a unique bounded strong solution $v$ to the VI \eqref{v_pb}. On $\Q^{+}_{\uc}$, it satisfies
\begin{align}\label{v}
\frac{\cc-\la\ell\ga}{r}\leq v \leq &~ \frac{\cc}{r},\\\label{vx}
0\leq v_x\leq &~\ell,\\\label{vxx}
-\frac{\la\ell}{\mu-\cc}\leq v_{xx}\leq &~ 0\quad{\rm a.e.,}\\\label{vc}
-\frac{\ell-1}{r}\leq v_c\leq &~ 0\quad{\rm a.e..}
\end{align}
Moreover, the solution does not relay on the choice of $\uc$.
\end{proposition}

Clearly, the solution given in \propref{prop:u} solves \eqref{v_pb00} on $\Q^{+}_{\uc}$, so by uniqueness it coincides with the solution of \eqref{v_pb00}.
Since the solution does not relay on the choice of $\uc$, we can uniquely extend it to the domain $\Q^{+}_{\infty}$, giving a solution to \eqref{v_pb00}.
Therefore, \propref{prop:u} implies \thmref{thm:u}. From now on, we focus on the construction of a solution to \eqref{v_pb}
and the proof of \propref{prop:u}.

\subsection{Approximation by a regime switching system}\label{sec:approximation}
We omit the proof of uniqueness for \eqref{v_pb} since it is similar to that of \lemref{lem:CP00}.
Due to uniqueness, one can solve \eqref{v_pb} backwardly along the $c$ direction.
In the rest part of this section, we will construct a strong solution to the local HJB equation \eqref{v_pb} with the desired properties in \propref{prop:u}.

Suppose the dividend payout rate can only take the following discrete values
$$c_i=\cc-(i-1)\Dc,\quad i=1,2,\cdots,n,$$
where $\Dc=\frac{1}{n}(\cc-\uc)$. We consider the following regime switching system,
\begin{align}\label{vi_pb}
\begin{cases}
\min\{\LL_{c_i} v_i-\TT v_i+h-c_i, \; v_i-v_{i-1}\}=0, & x\in\R^+,\quad i=1,2,\cdots,n,\medskip\\
v_0(x)=g (x), & x\in\R^+.
\end{cases}
\end{align}
Here, $v_i(x)$ can be regarded as an approximation of $v(x,c_i)$.

\begin{lemma}\label{lem:vi}
For $i=1,2,\cdots,n,$ the system \eqref{vi_pb} has a unique solution $v_i\in W^{1,\infty}(\R^+)$, which satisfies
\begin{align}\label{vi}
\frac{\cc-\la\ell\ga}{r}\leq &~ v_i\leq \frac{\cc}{r},\\\label{vix}
0\leq&~ v_i'\leq \ell.
\end{align}
\end{lemma}

{Lemma~\ref{lem:vi} establishes that the discrete approximation \eqref{vi_pb} is well-posed and that the solutions inherit the same qualitative properties as the conjectured value function: uniform bounds and $x$-Lipschitz continuity with constant $\ell$. These uniform estimates are essential for the subsequent limiting argument. These properties are inspired by \lemref{lem:lipchitz}.\medskip
} 

\begin{proofof}{\lemref{lem:vi}.}
The uniqueness is a consequence of \lemref{lem:CP1}.
We now show the existence.
From \thmref{lem:g} we know $v_0(x)=g(x) $ satisfies \eqref{vi} and \eqref{vix}.
We next prove the results by mathematical induction.

Suppose $v_{j-1}\in W^{1,\infty}(\R^+)$ (for some $1\leq j<n$) satisfies \eqref{vi} and \eqref{vix}, we are going to prove the VI in \eqref{vi_pb} admits a unique bounded solution $v_j\in W^{1,\infty}(\R^+)$ which still satisfies \eqref{vi} and \eqref{vix}.
This will complete the proof.

Consider the penalty approximation problem for $0<\ep<1$:
\begin{align}\label{vje_eq}
\LL_{c_j} v_j^\ep-\TT v_j^\ep+h-c_j+\beta_\ep(v_j^\ep-v_{j-1})=0\quad {\rm in}~~ \R^+,
\end{align}
where $\beta_\ep(\cdot)$ is a sequence of penalty functions indexed by $0<\ep<1$ satisfying
\begin{gather*}
\beta _{\ep}(\cdot)\in C^{\infty}(\R), \quad\beta _{\ep}(0)=-(\cc+\la^2\ell\ga/r+\la\ell\ga)<0, \quad\beta _{\ep}(x)=0\; \text{ for }\; x\geq \ep>0, \bigskip\\
\beta _{\ep
}(\cdot)\leq 0, \quad \beta _{\ep}^{\prime}(\cdot)\geq 0,\quad \beta _{\ep
}^{\prime \prime}(\cdot)\leq 0, \quad\lim\limits_{\ep \rightarrow 0+} \beta _{\ep}(x)=\begin{cases}
0, &{\rm if}\;\; x>0,\vspace{2mm} \\
-\infty, &{\rm if}\;\; x<0.
\end{cases}\end{gather*}%
The existence of a unique $C^1$-smooth bounded solution to the equation \eqref{vje_eq} can be obtained by constructing the approximation problem in bounded interval and applying the Leray-Schauder fixed point theorem (see \cite{Ev16} Theorem 4 on page 539), we leave it to the interested readers.

We come to prove
\begin{align}\label{vje}
v_{j-1} \leq v_j^\ep\leq \cc/r+\ep.
\end{align}
By the induction hypothesis, we have
\begin{align*}
0\leq c_j\leq \cc < \mu,\quad \frac{\cc-\la\ell\ga}{r}\leq v_{j-1}\leq \frac{\cc}{r},\quad 0\leq v_{j-1}'\leq \ell,\quad 0\leq h\leq \la\ell\ga,
\end{align*}
so
$$
\LL_{c_j} v_{j-1}-\TT v_{j-1}+h-c_j+\beta_\ep(v_{j-1}-v_{j-1})\leq (r+\la)\frac{\cc}{r}-\la \frac{\cc-\la\ell\ga}{r}+\la\ell\ga+\beta_\ep(0)=0.
$$
By \lemref{lem:CP1}, we obtain $v_j^\ep\geq v_{j-1}$.
Let $\Phi(x)=\cc/r+\ep$. Since $v_{j-1}\leq \cc/r$, we have
$$
\LL_{c_j} \Phi-\TT \Phi+h-c_j+\beta_\ep(\Phi-v_{j-1})\geq \cc+\ep r+h-c_j +\beta_\ep(\ep)\geq0.
$$
By \lemref{lem:CP1}, we obtain $v_j^\ep\leq \Phi$, completing the proof of \eqref{vje}.

We next prove
\begin{align}\label{vje'}
0\leq (v_j^\ep)' \leq \ell
\end{align}
Differentiating \eqref{vje_eq} yields
\begin{align*}
\LL_{c_j} \big((v_j^\ep)'\big)-\II \big((v_j^\ep)'\big)+\beta_\ep'(v_j^\ep-v_{j-1})(v_j^\ep)'=\la\ell (1-F)+\beta_\ep'(v_j^\ep-v_{j-1}) v_{j-1}' \quad {\rm in}~~ \R^+.
\end{align*}
Since $(v_j^\ep)'$, $v_{j-1}'$, $F$ and $\beta_\ep'(v_j^\ep-v_{j-1})$ are continuous and bounded in $\R^+$, so $(v_j^\ep)''$ is also continuous and bounded in $\R^+$.
For $\psi(x)=0$ and $\Psi(x)=\ell$, one has
\begin{align*}
\LL_{c_j} \psi-\II \psi+\beta_\ep'(v_j^\ep-v_{j-1})\psi=0 &\leq \la\ell (1-F)+\beta_\ep'(v_j^\ep-v_{j-1}) v_{j-1}' \\
&\leq r l+\la\ell (1-F)+\beta_\ep'(v_j^\ep-v_{j-1}) \ell\\
&=\LL_{c_j} \Psi-\II \Psi+\beta_\ep'(v_j^\ep-v_{j-1})\Psi.
\end{align*}
Applying \lemref{lem:CP1}, we get \eqref{vje'}.

Since $(v_j^\ep)_{0<\ep<1}$ is a bounded sequence in $W^{1,\infty}(\R^+)$, it has a subsequence converges to some $v_j \in W^{1,\infty}(\R^+)$ uniformly in $C([0,L])$ and weakly in $W^{1,\infty}([0,L])$ for any $L>0$. It is easy to check that $v_j $ is a $ W^{1,\infty}(\R^+)$ solution to
$$
\min\{\LL_{c_j} v_j-\TT v_j+h-c_j, \; v_j-v_{j-1}\}=0, \quad x\in\R^+.
$$
Moreover, by \lemref{lem:CP1}, this solution is unique in $ W^{1,\infty}(\R^+)$.
Finally, we can derive easily \eqref{vi} and \eqref{vix} from \eqref{vje} and \eqref{vje'} respectively.
\end{proofof}

From now on, we use $v_i$, $0\leq i\leq n$, to denote the unique solution to \eqref{vi_pb} in $W^{1,\infty}(\R^+)$. And define $\D_i=\{v_i > v_{i-1}\}$ for $1\leq i\leq n$.

\begin{lemma}\label{lem:vi''}
For $i=0,1,\cdots,n,$ the function $v_i$ satisfies
\begin{align}\label{vi''}
v_i''\geq-\frac{\la\ell}{\mu-\cc}
\end{align}
in the viscosity sense, which implies that $v_i(x)+\frac{1}{2} \frac{\la\ell}{\mu-\cc} x^2$ is a convex function.
\end{lemma}

{Lemma~\ref{lem:vi''} provides a uniform lower bound on the second derivative, independent of the discretization parameter $n$. Such concavity-type estimates are crucial for establishing the regularity of the limiting value function and for analyzing the free boundary in Section~\ref{sec:FB}.\medskip
}

\begin{proofof}{\lemref{lem:vi''}.}
By the equation \eqref{g'_eq} and the estimate \eqref{g_b} we have
\begin{align*}
(\mu-\cc)g''=&~ (r+\la ) g'-\II (g')-\la\ell (1-F) \geq 0-\la\ell F-\la\ell (1-F)=-\la\ell,
\end{align*}
so \eqref{vi''} holds for $i=0$.

Now suppose \eqref{vi''} holds for some $i=j-1$ ($1\leq j<n$).
Differentiating the equation in \eqref{vi_pb} and using the estimate \eqref{vix} we get
\begin{align*}
(\mu-c_j)v_j''=&~ (r+\la ) v_j'-\II (v_j')-\la\ell (1-F)
\geq 0-\la\ell F-\la\ell (1-F)=-\la\ell \quad{\rm in}~~ \D_j,
\end{align*}
which implies $v_j''\geq-\frac{\la\ell}{\mu-\cc}$ in $ \D_j$.
On the other hand, since $v_j-v_{j-1}$ attains its minimum value 0 in $\R\setminus\D_j$, we have $v_j''\geq v_{j-1}''\geq-\frac{\la\ell}{\mu-\cc}$ in $\R\setminus\D_j$ by the induction hypothesis, so \eqref{vi''} also holds for $i=j$. This completes the proof.
\end{proofof}

Define
$$u_i:=\frac{v_i-v_{i-1}}{\Dc },\quad i=1,2,\cdots,n.$$

\begin{lemma}\label{lem:ui_b}
For $i=1,2,\cdots,n$, we have
\begin{align}\label{ui_b}
0\leq u_i\leq \frac{\ell-1}{r}.
\end{align}
\end{lemma}

{The quantity $u_i$ approximates $-v_c$ (the partial derivative with respect to the dividend rate) in the discrete setting. The uniform bound \eqref{ui_b} shows that the value function's sensitivity to changes in $c$ is bounded, a property that will be preserved in the limit and is essential for verifying the HJB conditions.\medskip
}

\begin{proofof}{\lemref{lem:ui_b}.}
We only need to prove the second inequality.
Let $\ov_i=v_{i-1}+\frac{\ell-1}{r} \Dc $, since $v_{i-1}'\leq \ell,$ we have
\begin{align*}
\LL_{c_i} \ov_i-\TT \ov_i+h-c_i
=&\LL_{c_i} v_{i-1}-\TT v_{i-1}+h-c_i+(\ell-1) \Dc\\
=&\LL_{c_{i-1}} v_{i-1}-\TT v_{i-1}+h-c_{i-1}+(\ell-v_{i-1}') \Dc\geq 0.
\end{align*}
Applying \lemref{lem:CP1} we get $v_i\leq \ov_i$, which implies \eqref{ui_b} holds.
\end{proofof}

\begin{lemma}\label{lem:uix_b}
For $i=1,2,\cdots,n$, we have
\begin{align}\label{uxi_b}
-\frac{(r+\la)(\ell-1)}{r(\mu-\cc)}\leq u_i' \leq \frac{(r+\la)(\ell-1)+r}{r(\mu-\cc)} \quad{\rm a.e.\;in}~~ \D_{i-1}.
\end{align}
\end{lemma}

{This lemma controls the spatial derivative of the divided differences on the region where the dividend rate is actively increasing. These bounds are critical for establishing the differentiability with respect to $x$ of the limiting value function and the continuity of the free boundary later. \medskip
}

\begin{proofof}{\lemref{lem:uix_b}.}
Since $u_i=0$ in $\R\setminus \D_i$, we have $u_i'=0$ in the inner set of $\R\setminus\D_i$, note that $\D_i$ is an open set, so its boundary set is a countable set whose Lebesgue measure is 0, therefore, we have $u_i'=0$ a.e. in $\R\setminus \D_i$ satisfying \eqref{uxi_b}. We only need to prove \eqref{uxi_b} holds a.e. in $\D_i\cap \D_{i-1}.$

The difference between the equations of $v_i$ and $v_{i-1}$ in $\D_i\cap \D_{i-1}$ satisfies
\begin{align*}
\LL_{c_{i-1}} u_i-\TT u_i+1-v_i'=0 \quad{\rm a.e.\;in}~~ \D_i\cap \D_{i-1}.
\end{align*}
Applying \eqref{vix}, \eqref{ui_b} and $u_i\geq 0$, we get
\begin{align*}
(\mu-c_{i-1}) u_i'=&~(r+\la) u_i-\TT u_i+1-v_i'\medskip\\
\in &~ \Big[-(\frac{\la}{r}+1)(\ell-1), (r+\la)\frac{\ell-1}{r}+1 \Big] \quad{\rm in}~~ \D_i\cap \D_{i-1},
\end{align*}
which implies \eqref{uxi_b} holds a.e. in $\D_i\cap \D_{i-1}.$
The proof is complete.
\end{proofof}

\begin{lemma}\label{lem:uci_b}
For $i=2,3, \cdots,n$, we have
\begin{align}\label{ui_b2}
u_{i-1}\leq u_i+B \Dc,
\end{align}
where
\begin{align*}
B=\frac{2(r+\la)(\ell-1)}{r^2(\mu-\cc)}.
\end{align*}
\end{lemma}

{ This lemma establishes a one-side bound for the second-order difference of the solution with respect to $c$, which is crucial for proving the continuity of the free boundary. \medskip
}

\begin{proofof}{\lemref{lem:uci_b}.}
Since $u_{i-1}=0$ in $\R\setminus \D_{i-1}$ and $u_i\geq 0$, we only need to prove \eqref{ui_b2} holds in $\D_{i-1}$.
Notice
\begin{align*}
\LL_{c_{i-1}} v_{i-1}-\TT v_{i-1}+h-c_{i-1} &=0 \quad{\rm in}~~ \D_{i-1},\medskip\\
\LL_{c_i} v_i-\TT v_i+h-c_i &\geq 0 \quad{\rm a.e.\;in}~~ \D_{i-1},\medskip\\
\LL_{c_{i-2}} v_{i-2}-\TT v_{i-2}+h-c_{i-2} &\geq 0 \quad{\rm a.e.\;in}~~ \D_{i-1},
\end{align*}
so we have
\begin{align*}
\LL_{c_{i-1}} u_i-\TT u_i+1-v_i' &\geq 0 \quad{\rm a.e.\;in}~~ \D_{i-1},\medskip\\
\LL_{c_{i-2}} u_{i-1}-\TT u_{i-1}+1-v_{i-1}' &\leq 0 \quad{\rm a.e.\;in}~~ \D_{i-1},
\end{align*}
Let $\varphi=(u_i-u_{i-1})/\Dc+B$, then
\begin{align*}
\LL_{c_{i-2}} \varphi-\TT \varphi \geq 2 u_i'+r B\geq 0\quad{\rm a.e.\;in}~~ \D_{i-1},
\end{align*}
where the last inequality is due to the first inequality of \eqref{uxi_b}.
Moreover, noting that $\varphi \geq 0$ in $\R^+\setminus\D_{i-1}$, so by \lemref{lem:CP1} we have $\varphi \geq 0$ in $\D_{i-1}$. Consequently, \eqref{ui_b2} is established.
\end{proofof}

For each positive integer $N$, let $v^N_i$, $i=1,\cdots,2^N$ be the solution to \eqref{vi_pb} with $n=2^N$ and $\Dc=2^{-N}(\cc-\uc)$. Let $c_i^N=\cc-i \Dc$.
Since $(v^{N}_i)'\leq \ell$, $v^{N}_i$ satisfies
\begin{align}\label{viN_pb}
\begin{cases}
\min\{\LL_{c_i^N} v_i^N-\TT v_i^N+h-c_i^N, \;\ell-(v_i^N)', \; v_i^N-v_{i-1}^N\}=0, & x\in\R^+,\quad i=1,2,\cdots,2^N,\medskip\\
v_0^N (x)=g (x), & x\in\R^+,
\end{cases}
\end{align}
which is the HJB equation connected to the value function w.r.t. finite ratcheting strategy:
\begin{align*}
v_i^N(x)=
\sup\limits_{S^{x,i,N}\in \Pi_{x,i}^N} J(S^{x,i,N}),
\end{align*}
where
\begin{align*}
\Pi_{x,i}^N=\Big\{ \{(C_t, D_t)\}_{t\geq 0} \in \Pi_{x,c_i^N} \;\Big|\; C_t\in \{c_i^N,\;c_{i-1}^N,\cdots, c_0^N\}, ~~ t\geq 0\Big\},
\end{align*}
and for $S^{x,i,N}=\{(C_t, D_t)\}_{t\geq 0}\in \Pi_{x,c_i^N}$, we define
\begin{align*}
J(S^{x,i,N})=\E \Bigg[\int_0^\infty e^{-r t} C_t \dt-\ell \int_0^\infty e^{-r t} \d D_t\Bigg].
\end{align*}

\begin{lemma}\label{lem:vN1}
For any $x\in \R^+$ and $i=1,2,\cdots,n$, we have
\begin{align}\label{vN1}
v^{N}_{i}(x) \leq v^{N+1}_{2i}(x).
\end{align}
\end{lemma}
\begin{proof}
Note that $\Pi_{x,i}^N \subseteq \Pi_{x,2i}^{N+1}$, we have
$$
\sup\limits_{S^{x,i,N}\in \Pi_{x,i}^N} J(S^{x,i,N})\leq \sup\limits_{S^{x,2i,N+1}\in \Pi_{x,2i}^{N+1}} J(S^{x,2i,N+1}), 
$$
which is \eqref{vN1}.
\end{proof}

\begin{lemma}\label{lem:vN2}
For any $x_1,\;x_2\in \R^+$ and every $i=1,2,\cdots,n$, we have
\begin{align}\label{vN2}
\frac{v^{N}_i(x_1)+v^{N}_i(x_2)}{2} \leq v^{N+1}_{2i}\Big(\frac{x_1+x_2}{2}\Big).
\end{align}
\end{lemma}

{
Inequalities \eqref{vN1} and \eqref{vN2} reflect fundamental properties of the discrete approximations.
Inequality \eqref{vN2} is a convexity-type property inherited from the concavity of the value function in the surplus variable (Lemma~\ref{lem:lipchitz}) and the linearity of the dynamics.
Passing to the limit yields
the concavity of $v(\cdot,c)$. This concavity is essential for the free boundary analysis in Section~\ref{sec:FB}, where it guarantees the existence of a well-defined switching boundary $\sw(\cdot)$.\medskip\\
}
\begin{proofof}{\lemref{lem:vN2}.}
For any $\ep > 0$, let $S^{x_1,i,N}_\ep=\{(C_t^{x_1}, D_t^{x_1})\}_{t\geq 0}$ with $C_{0-}^{x_1}=c_i^N$ (resp., $S^{x_2,i,N}_\ep=\{(C_t^{x_2}, D_t^{x_2})\}_{t\geq 0}$ with $C_{0-}^{x_2}=c_i^N$) an $\ep$-optimal strategy corresponding to initial surplus $x_1$ (resp., $x_2$) such that
$$
J(S^{x_1,i,N}_\ep) > v^{N}_i(x_1)-\ep, \quad J(S^{x_2,i,N}_\ep) > v^{N}_i(x_2)-\ep.
$$
Then, the strategy $\frac{1}{2} S^{x_1,i,N}_\ep+\frac{1}{2} S^{x_2,i,N}_\ep=\{(\frac{1}{2}C_t^{x_1}+\frac{1}{2} C_t^{x_2}, \frac{1}{2} D_t^{x_1}+\frac{1}{2} D_t^{x_2})\}_{t\geq 0}$ with $\frac{1}{2} C_{0-}^{x_1}+\frac{1}{2} C_{0-}^{x_2}=c_i^N=c_{2i}^{N+1} $ belongs to $\Pi_{(x_1+x_2)/2,\;2i}^{N+1}$. Therefore
\begin{align*} 
v^{N+1}_{2i}\Big(\frac{x_1+x_2}{2}\Big)\geq &~ J\Big(\frac{1}{2} S^{x_1,i,N}_\ep+\frac{1}{2} S^{x_2,i,N}_\ep\Big)\medskip\\
=&~\frac{1}{2}J( S^{x_1,i,N}_\ep )+\frac{1}{2}J(S^{x_2,i,N}_\ep)> \frac{1}{2}v^{N}_i(x_1)+\frac{1}{2} v^{N}_i(x_2)-\ep.
\end{align*}
Then letting $\ep\to 0$ we get \eqref{vN2}.
\end{proofof}


\subsection{Construction of the solution to \eqref{v_pb}}\label{sec:solvability}

Now we are ready to prove \propref{prop:u}. 
As earlier mentioned, we only need to construct a bounded strong solution to the local HJB equation \eqref{v_pb}.

Let $v^N(x,c)$ be the linear interpolation function of $v^N_i(x)$.
Then \eqref{vi}, \eqref{vix} and \eqref{ui_b} imply $v^N$ is uniform bounded and uniform Lipschitz continuous in $\Q^{+}_{\uc}$. Apply the Arzela-Ascoli theorem, there exists
a Lipschitz continuous function $v$ in $\Q^{+}_{\uc}$,
and a subsequence $v^{N_k}$ such that, for each $L>0$,
$v^{N_k}\longrightarrow v $ in $C([0,L]\times [\uc,\cc]).$
Moreover, \eqref{v}, \eqref{vx}, the first inequality of \eqref{vxx}, \eqref{vc} are derived from \eqref{vi}, \eqref{vix}, \eqref{vi''} and \eqref{ui_b}, respectively.

Now, we prove the second inequality of \eqref{vxx}. For any fixed $x_1,\;x_2\in \R^+$ and $c=\cc-i_0 2^{-N_{k_0}}(\cc-\uc)$ with $i_0$, $k_0\in \N$. Let $i_{k}=2^{N_{k}-N_{k_0}}i_0$ for $k\geq k_0$. Then $c=\cc-i_k 2^{-N_{k}}(\cc-\uc)$ for any $k\geq k_0$. Thanks to \eqref{vN2} and \eqref{vN1}, we have
\begin{align*}
\frac{v^{N_k}_{ i_k}(x_1)+v^{N_k}_{ i_k}(x_2)}{2} \leq v^{N_k+1}_{{ 2i_{k}}}\Big(\frac{x_1+x_2}{2}\Big) \leq v^{N_{k+1}}_{{ i_{k+1}}}\Big(\frac{x_1+x_2}{2}\Big).
\end{align*}
Letting $k\to \infty$ gives
\begin{align*}
\frac{v(x_1,c)+v(x_2,c)}{2} \leq v\Big(\frac{x_1+x_2}{2},c\Big).
\end{align*}
Since $v$ is continuous in $\Q^{+}_{\uc}$, we conclude $v(\cdot,c)$ is concave for each $c\in[\uc,\cc]$.

In addition, \eqref{v}-\eqref{vxx} imply $v(\cdot,c)\in W^{2,\infty}(\R^+)\subseteq C^1(\R^+)$ for each $c\in [\uc,\cc]$, so we have $v\in \A_{\uc}$.

We come to prove $v$ satisfies
the third and forth properties in \defref{def:solution}.
For each $c\in [\uc,\cc]$, by the construction of $v$, there exists $c^k=\cc-i_k 2^{-N_k}(\cc-\uc)$ such that $c^k\to c$ and
$v^{N_k}_{i_k}(\cdot)\longrightarrow v(\cdot,c)$ in $C[0,L]$ for any $L>0$.
Moreover, from \eqref{vix} we also have
$$v^{N_k}_{i_k}(\cdot)\hbox{(or its subsequence)}\longrightarrow v(\cdot,c)\quad\hbox{ weakly in }W^{1,\infty}([0,L])$$
for any $L>0$.
Letting $k\to \infty$ in the inequality
$\LL_{c^k} v^{N_k}_{i_k}-\TT v^{N_k}_{i_k}+h-c^k \geq 0$, we get \eqref{-lv>=02}.

On the other hand, if $v(x,c)>v(x,s)$ for any $c<s\leq \cc$, then for any $n\in \N$, there exists a sufficiently large $k_n\in \N$ and $s^{k_n}=\cc-j_{k_n} 2^{-N_{k_n}}(\cc-\uc)\in [c,c+1/n]$ such that $v^{N_{k_n}}_{j_{k_n}}(x)> v^{N_{k_n}}_{j_{k_n}-1}(x)$. As a consequence,
$$
\Big(\LL_{s^{k_n}} v^{N_{k_n}}_{j_{k_n}}-\TT v^{N_{k_n}}_{j_{k_n}}+h\Big)(x)-s^{k_n}=0,
$$
i.e.
$$
(v^{N_{k_n}}_{j_{k_n}})'(x)=\frac{1}{\mu-s^{k_n} }\Big((\la+r) v^{N_{k_n}}_{j_{k_n}}-\TT v^{N_{k_n}}_{j_{k_n}}+h\Big)(x)-s^{k_n}.
$$
Denote $K_n=(v^{N_{k_n}}_{j_{k_n}})'(x)$. Since $s^{k_n}\to c$ and $v^{N_{k_n}}_{j_{k_n}}(\cdot)\to v(\cdot,c)$ in $C([0,x])$ when $n\to \infty$, we have
$$
\lim\limits_{n\to \infty} K_n=\frac{1}{\mu-c }\Big((\la+r) v(x,c)-\TT v(x,c)+h(x)\Big)-c.
$$
Moreover, due to \eqref{vi''}, we have
$$
v^{N_{k_n}}_{j_{k_n}}(y)\geq v^{N_{k_n}}_{j_{k_n}}(x)+K_n ( y-x )-\frac{1}{2} \frac{\la\ell}{\mu-\cc} ( y-x )^2,\quad ~ y\in \R^+.
$$
Letting $n\to \infty$ in the above inequality yields
$$
v(y,c)\geq v(x,c)+( \lim\limits_{n\to \infty} K_n) ( y-x )-\frac{1}{2} \frac{\la\ell}{\mu-\cc} ( y-x )^2,\quad ~ y\in \R^+.
$$
This implies
$$
v_x(x,c)=\lim\limits_{n\to \infty} K_n=\frac{1}{\mu-c }\Big((\la+r) v(x,c)-\TT v(x,c)+h(x)\Big)-c,
$$
and \eqref{-lv=02} follows.
The proof of \propref{prop:u} is complete.

Since \propref{prop:u} implies \thmref{thm:u}, we complete the proof of \thmref{thm:u} as well.
From now on we use $v$ to denote the solution given in \thmref{thm:u}.


\section{On the free boundary $\sw(\cdot)$ and the equivalent maximum rate $\swb(\cdot,\cdot)$}\label{sec:FB}
In this section, we prove Proposition \ref{proposition:freeboundary} and \lemref{lemma:swb}.
\subsection{Separated regions $\SS$ and $\NS$}
To this end, we need the following technical result.
\begin{lemma}\label{lem:N}
For any $(x_0,c_0)\in \NS\bigcup(\R^+\times \{\cc\})$, if there exists $d_0<c_0$ such that $v(x_0,d_0)=v(x_0,c_0)$, then we have
\begin{align}\label{vx<=1}
(c_0-c)(1-v_x(x_0,c_0))-[\TT v(x_0,c)-\TT v(x_0,c_0)]\geq 0,\quad ~ c\in [d_0,c_0),
\end{align}
and
\begin{align}\label{SS}
v(x,c)=v(x,c_0),\quad ~ (x,c)\in [x_0,+\infty)\times[d_0,c_0].
\end{align}
\end{lemma}
\begin{proof}
For any $c\in [d_0,c_0)$, note that $v(\cdot,c)-v(\cdot,c_0)$ attains its minimum value 0 at $x_0>0$, so we have
$v_x(x_0, c)=v_x(x_0, c_0).$
Combining
$$\LL_{c} v(x_0, c)-\TT v(x_0, c)+h(x_0)-c \geq 0$$
and
$$\LL_{c_0} v(x_0, c_0)-\TT v(x_0, c_0)+h(x_0)-c_0=0$$
we get \eqref{vx<=1}.

Now, we prove \eqref{SS}. Note that $v(x,c)$ satisfies the following problem on $u$ in the domain $[x_0,+\infty)\times[d_0,c_0]$,
\begin{align}\label{v_pb1}
\begin{cases}
\min\{\LL_c u-\KK u+f(x,c)-c, \;-u_c\}=0, & (x,c)\in [x_0,+\infty)\times[d_0,c_0],\medskip\\
u(x,c_0)=v(x,c_0), & x>x_0,
\end{cases}
\end{align}
where
\begin{align*}
f(x,c)&=-\la \int_{x-x_0}^x v(x-y,c)\d F(y)-\la v(0,c) (1-F(x))+h(x),
\end{align*}
and $\KK$ is a linear operator,
defined as
$$
\KK u(x)=\la \int_{0}^{x-x_0} u (x-y,c)\d F(y).
$$

On the other hand, let $w(x,c)=v(x,c_0)$, $(x,c)\in [x_0,+\infty)\times[d_0,c_0]$, then $w_c=0$ in $[x_0,+\infty)\times[d_0,c_0]$. For any $(x,c)\in [x_0,+\infty)\times[d_0,c_0]$ we have,
\begin{align*}
&~\big(\LL_c w-\KK w+f \big)(x,c)-c \medskip\\
=&~ \big(\LL_{c_0} v-\KK v+f \big)(x,c_0)-c_0+(c_0-c)(1-v_x(x,c_0))+(f(x,c)-f(x,c_0))\medskip\\
\geq &~ (c_0-c)(1-v_x(x,c_0))+(f(x,c)-f(x,c_0))\medskip\\
=&~ (c_0-c)(1-v_x(x,c_0))\medskip\\
&~-\la \int_{0}^{x_0} [v(z,c)-v(z,c_0)] p(x-z) \dz-\la [v(0,c)-v(0,c_0)] (1-F(x))\medskip\\
\geq &~ (c_0-c)(1-v_x(x_0,c_0)) \medskip\\
&~-\la \int_{0}^{x_0} [v(z,c)-v(z,c_0)] p(x_0-z) \dz-\la [v(0,c)-v(0,c_0)] (1-F(x_0))\medskip\\
=&~ (c_0-c)(1-v_x(x_0,c_0))-[\TT v(x_0,c)-\TT v(x_0,c_0)]\\
\geq&~ 0,
\end{align*}
where the second inequality is due to $v_x(x,c_0)$, $p(x-z)$, $1-F(x)$ are non-increasing in $x$ and $v(z,c)-v(z,c_0)\geq 0$ for all $z\in [0,x_0]$, and the last inequality is due to \eqref{vx<=1}.
So $w$ also satisfies \eqref{v_pb1}.
Similar to the proof of \lemref{lem:CP00}, we can prove the solution to the problem \eqref{v_pb1} is unique, so we have
$v(x,c)=w(x,c)=v(x,c_0)$, $(x,c)\in [x_0,+\infty)\times[d_0,c_0]$.
\end{proof}

The following result shows that $\SS$ and $\NS$ are separated by $\sw(\cdot)$.
\begin{lemma}\label{lem:N1}
If $(x,c)\in \SS$, then $(y,c)\in \SS$ for all $y\geq x$.
\end{lemma}
\begin{proof}
Denote $c_0=\sup\{s\in (c,\cc]\mid v(x,s)=v(x,c)\}$. Since $(x,c)\in \SS$, we know $c_0>c$. Note that $(x,c_0)\in \NS \bigcup(\R^+\times \{\cc\})$, apply \lemref{lem:N} we have $v(y,s)=v(y,c)$ for all $(y,s)\in [x,+\infty)\times [c,c_0]$, which implies $(y,c)\in \SS$ for all $y\geq x$.
\end{proof}

\subsection{Properties of the free boundary $\sw(\cdot)$}
In this section we prove that the curve $\sw(\cdot)$ is continuous on $(-\infty, \cc]$.

\begin{lemma}\label{lem:x=0}
For any $c_0\in(-\infty, \cc)$, the following three claims are equivalent.
\begin{enumerate}
\item $\sw(c_0)=0.$
\item $v_x(0,c_0)\leq 1.$
\item $\sw(c)=0$ for all $c\in (-\infty,c_0]$.
\end{enumerate}
\end{lemma}
\begin{proof}
Trivially, the third claim implies the first one.
Noting that \eqref{vx<=1} implies
$v_x\leq 1$ in $\SS$
and thus for each $c\in (-\infty, \cc),$ we have
$v_x(\sw(c),c)\leq 1.$
Hence, the first claim implies the second one as well.

Now, suppose the second claims holds, then we only need to show the third claim. From \eqref{vxx} we know $v_x(\cdot,c_0)$ is non-increasing, so $v_x(x,c_0)\leq 1$ for all $x\in\R^+$. Let $u(x,c)=v(x,c_0)$. Then $u$ satisfies
$u_c=0$ and
$$
\big(\LL_c u-\TT u+h\big)(x,c)-c=\big(\LL_{c_0} v-\TT v+h\big)(x,c_0)-c_0+(c_0-c)(1-v_x(x,c_0))\geq 0,
$$
for any $(x,c)\in \R^+\times(-\infty,c_0]$,
so $u$ is a strong solution to
\begin{align}\label{u_pb1}
\begin{cases}
\min\{\LL_c u-\TT u+h-c, \;-u_c\}=0, & x\in\R^+,\; c\in (-\infty,c_0],\medskip\\
u(x,c_0)=v(x,c_0), & x\in\R^+.
\end{cases}
\end{align}
On the other hand, $v(x,c)$ also satisfies this problem, by the uniqueness of solution, we prove $v(x,c)=v(x,c_0)$ for all $(x,c)\in \R^+\times(-\infty,c_0]$, which implies the third claim.
\end{proof}

In the following we show that $\sw(\cdot)$ is continuous on $(-\infty, \cc]$.
To prove the continuity, we need the following estimate.
\begin{lemma}\label{lem:ucL}
For $ c_*<c_*+\Dc_*\leq c^*-\Dc^*<c^*\leq \cc$, we have
\begin{align}\label{ucL}
\frac{v(x,c^*-\Dc^*)-v(x,c^*)}{\Dc^*}\leq \frac{v(x,c_*)-v(x,c_*+\Dc_*)}{\Dc_*}+B(c^*-c_*),
\end{align}
where $B$ is given in \lemref{lem:uci_b}.
\end{lemma}
\begin{proof}
Denote $$u^N_i=\frac{v^N_i-v^N_{i-1}}{\Dc^N},$$
where $v^N_i(x),\;i=1,\cdots,2^N$ is the solution to \eqref{vi_pb} and $\Dc^N=(\cc-\underline{c})/(2^N)$ for a fixed $\underline{c}<c_*$. \lemref{lem:uci_b} implies
\begin{align*}
u^N_{i-k}\leq u^N_{i+j+l+1}+(k+j+l+1) B \Dc^N
\end{align*}
for $k,l,j\in \N$, so
\begin{align*}
\frac{1}{n}\sum_{k=0}^{n-1} u^N_{i-k}\leq \frac{1}{m}\sum_{l=0}^{m-1} u^N_{i+j+l+1}+(n+m+j+1) B \Dc^N,
\end{align*}
for $n,m\in \N$, i.e.
\begin{align*}
\frac{v^N_{i}-v^N_{i-n}}{n \Dc^N}\leq \frac{v^N_{i+j+m}-v^N_{i+j}}{m \Dc^N}+(n+m+j+1) B \Dc^N.
\end{align*}
Denote $c^N_i=\cc-i \Dc^N$ and suppose
$$c_*\in [c^N_{i+j+m+1},c^N_{i+j+m}],\; c_*+\Dc_*\in [c^N_{i+j},c^N_{i+j-1}],\; c^*-\Dc^*\in [c^N_{i},c^N_{i-1}],
\; c^*\in [c^N_{i-n+1},c^N_{i-n}],$$
for suitable $j,n,m\in \N$. 
Since $v^N(x,c)$(the linear interpolation function of $v^N_i(x)$) is non-increasing w.r.t. $c$, we have
\begin{align*}
&\frac{v^N(x,c^*-\Dc^*)-v^N(x,c^*)}{\Dc^*}
\leq\frac{v^N(x,c^N_{i})- v^N(x,c^N_{i-n})}{(n-2) \Dc^N}\medskip\\
\leq& \frac{n}{n-2}\Bigg[\frac{v^N(x,c^N_{i+j+m})-v^N(x, c^N_{i+j})}{m \Dc^N}+(n+m+j+1) B \Dc^N\Bigg]\medskip\\
\leq& \frac{n}{n-2}\Bigg[\frac{m+2}{m}\cdot\frac{v^N(x,c_*)-v^N(x, c_*+\Dc_*)}{\Dc_*}+\frac{n+m+j+1}{n+m+j-1}B (c^*-c_*)\Bigg].
\end{align*}
Sending $N\to \infty$, and observing that $n$, $m$, $i$, and $j$ also tend to infinity simultaneously, we deduce \eqref{ucL} and complete the proof.
\end{proof}

\begin{lemma}\label{thm:xcccontinuous}
The curve $\sw(\cdot)$ is continuous in $(-\infty, \cc)$.
\end{lemma}
\begin{proof}
We fix an arbitrary $c\in (-\infty, \cc)$.
We first prove
\begin{align}\label{x_lim}
\liminf\limits_{s\to c+}\sw(s)=\limsup\limits_{s\to c+}\sw(s).
\end{align}

Suppose on the contrary, there exist $x_*$, $x^*$ such that
$$\liminf\limits_{s\to c+}\sw(s)<x_*<x^*<\limsup\limits_{s\to c+}\sw(s).$$
Note there exist a sequence $(c_n, \Dc_n) \to (c+, 0+)$ such that
\begin{align}\label{cn}
v(x,c_n+\Dc_n)=v(x,c_n),\quad x\geq x_*,
\end{align}
and a sequence $s_1>s_2>\cdots>s_n\to c$ such that
\begin{align}\label{lv=0}
\LL_{s_n} v(x,s_n)-\TT v(x,s_n)+h(x)-s_n=0,\quad x< x^*.
\end{align}
Let $$u^n(x)=\frac{v(x,s_n)-v(x,s_{n-1})}{s_{n-1}-s_n}\geq 0.$$
Due to \eqref{vc}, $u^n$ is bounded.
Note $u^n(x)$ satisfies
\begin{align}\label{uDsn}
\LL_{s_n} u^n(x)-\TT u^n(x)=v_x (x,s_{n-1})-1,\quad x\in[0,x^*],
\end{align}
so we further know $u^n$ is bounded in $W^{1,\infty}([0,x^*])$. Therefore, it has a subsequence converges to some $u\in W^{1,\infty}([0,x^*])$ weekly in $W^{1,\infty}([0,x^*])$ and uniformly in $C([0,x^*])$.

Differentiating both sides of \eqref{uDsn} w.r.t. $x$ gives
\begin{align*}
\LL_{s_n} (u^n)'(x)-\II (u^n)'(x)=v_{xx} (x,s_{n-1})\leq 0,\quad x\in(0,x^*),
\end{align*}
and note that $ (u^n)'(x)=0$ for $x\geq x^*$, by \lemref{lem:CP1} we have $ (u^n)'\leq 0$ in $\R^+$. It hence follows $u'\leq 0$ in $(0,x^*)$.

On the other hand, applying \lemref{lem:ucL} and \eqref{cn} yields
$0\leq u^n(x)\leq B(s_n-c)\longrightarrow 0$ for $x>x_*,$
so $u(x)=0$ for all $x\in [x_*,x^*]$.
Moreover, \eqref{vxx} implies $v (\cdot,s_{n-1})$ is bounded in $W^{2,\infty}([0,x^*])$. Since $v (\cdot,s_{n-1})$ converges to $v (\cdot,c)$ in $C([0,x^*])$, it has a subsequence converges to $v(\cdot,c)$ weekly in $W^{2,\infty}([0,x^*])$ and uniformly in $C^1([0,x^*])$.
Then letting $n\to \infty$ (along a suitable subsequence) in \eqref{uDsn}, we get
\begin{align}\label{express1}
v_x (x,c)=-\la \int_0^{x} u(x-z) p(z) \dz-\la u(0) (1-F(x))+1,\quad x\in(0,x^*).
\end{align}
Differentiating both sides w.r.t. $x$ gives
\begin{align*}
v_{xx}(x,c)=-\la \int_0^{x} u'(x-z) p(z) \dz,\quad x\in(x_*,x^*).
\end{align*}
Since $u'\leq 0$, $p>0$ and $v_{xx}\leq 0$, we conclude $u'=0$ in $(0,x^*)$.
Combining $u=0$ in $(x_*,x^*)$, we get $u=0$ in $(0,x^*)$.
It implies $v_x(\cdot,c)=1$ by \eqref{express1}, so $v(\cdot,c)=x+a_0$ in $(0,x^*)$ for some constant $a_0$.
Sending $n\to\infty$ in \eqref{lv=0} and using the above two expressions, we obtain
$$-(\mu-c)+(r+\la)(x+a_0)-\la \int_0^x (x-y)p(y)\dy-\la a_0+h(x)-c=0,\quad x\in(0,x^*).$$
Differentiating both sides w.r.t. $x$ twice gives
$ \la(\ell-1)p(x)=0$ for $x\in(0,x^*),$
which contracts \eqref{C1}. So the claim \eqref{x_lim} follows.

In a similar way, we can prove $\liminf\limits_{s\to c-}\sw(s)=\limsup\limits_{s\to c-}\sw(s)$ and $\liminf\limits_{s\to c-}\sw(s)\geq \limsup\limits_{s\to c+}\sw(s).$
Thus
$$\lim\limits_{s\to c-}\sw(s)\geq \lim\limits_{s\to c+}\sw(s).$$
To prove the reverse inequality, we suppose, on the contrary, there exist $x_*, x^*$ such that $$\lim\limits_{s\to c+}\sw(s)<x_*<x^*<\lim\limits_{s\to c-}\sw(s).$$
Since $\sw(\cdot)\geq 0$, we have $x_*, x^*>0$. Notice
\begin{align}\label{between2x}
\LL_{c} v(x,c)-\TT v(x,c)+h(x)-c=0,\quad x\in [x_*, x^*].
\end{align}
Since $\lim\limits_{s\to c+}\sw(s)<x_*$, there exists $\ep\in (0,\cc-c)$ such that $\sw(s)<x_*$ for all $s\in (c,c+\ep)$ and thus $[x_*,\infty)\times(c,c+\ep)\subset \SS$, by the continuity of $v$ we have $(x_*,c)\in \SS$. 
Let $$c_0=\sup\{c'\in[c,\cc]\mid v(x_*,c)=v(x_*,c')\}.$$
Then we have either $c<c_0<\cc$, $(x_*,c_0)\in \NS$ or $c_0=\cc$.
Due to \lemref{lem:N} we have
\begin{align}\label{ineq1}
(c_0-c)(1-v_x(x_*,c_0))-[\TT v(x_*,c)-\TT v(x_*,c_0)]\geq 0,~~~ v(x,c)=v(x,c_0),~x\geq x_*.
\end{align}
Thus, for $x\geq x_*$,
\begin{align*}
&\TT v(x,c)-\TT v(x,c_0)\medskip\\
=&~\la \int_0^x [v(z,c)-v(z,c_0)]p(x-z) \dz+\la [v(0,c)-v(0,c_0)] (1-F(x))\medskip\\
=&~\la \int_0^{x_*} [v(z,c)-v(z,c_0)]p(x-z) \dz+\la [v(0,c)-v(0,c_0)] (1-F(x)).
\end{align*}
If $v(0,c)-v(0,c_0)=0$, then $\sw(c)=0$. By \lemref{lem:x=0}, it follows $\sw(s)=0$ for all $s\in (-\infty,c]$, which contradicts $\lim\limits_{s\to c-}\sw(s)>x^*>0$. So we have $v(0,c)-v(0,c_0)>0$. Combining $F'>0$ and $p'\leq0$, we have $\TT v(x,c)-\TT v(x,c_0)$ is strictly decreasing in $x\geq x_*$.
Hence, by \eqref{ineq1}, we have
\begin{align*}
&~\big(\LL_{c} v(x,c)-\TT v(x,c)+h(x)-c\big)\\
&~-\big(\LL_{c_0} v(x,c_0)-\TT v(x,c_0)+h(x)-c_0\big)\\
=&~ (c-c_0) v_x(x,c_0)+\TT v(x,c_0)-\TT v(x,c)+c_0-c>0, \quad x> x_*,
\end{align*}
contradicting to \eqref{v_pb} and \eqref{between2x}.
We now conclude
$\lim\limits_{s\to c-}\sw(s)=\lim\limits_{s\to c+}\sw(s).$

It is only left to prove
$\lim\limits_{s\to c+}\sw(s)=\sw(c).$
Firstly, for any $x>\sw(c),$ by the definition of $\sw(c)$ and \lemref{lem:N}, there exists $\os>c$ such that $$v(y,s)=v(y,c),~~~ (y,s)\in [x,+\infty)\times[c,\os].$$
This implies $\sw(s)\leq x$ for $s\in [c,\os)$. Consequently,
it follows $\lim\limits_{s\to c+}\sw(s)\leq \sw(c)$.
On the other hand, for each $x>\lim\limits_{s\to c+}\sw(s),$ there exists $\os>c$ such that $x>\sw(s)$ for $s\in (c,\os]$. Hence, $[x,+\infty)\times(c,\os]\subseteq\SS$, and $v(x,s)=v(x,\os)$ for $s\in (c,\os].$ By the continuity of $v$, it follows $v(x,c)=v(x,\os)$, which implies $\sw(c)\leq x$. Hence, $\sw(c)\leq\lim\limits_{s\to c+}\sw(s).$ The proof is complete.
\end{proof}

\begin{lemma}\label{thm:xccbounded}
The curve $\sw(c)$ is bounded near $\cc$.
\end{lemma}
\begin{proof}
Suppose $\sw$, on the contrary, is not bounded near $\cc$. Since it is continuous, there exists a strictly increasing sequence $c_n\to \cc$ such that $\sw(c_n)\to+\infty$.

Let
$$u^n(x)=\frac{v(x,c_n)-v(x,\cc)}{\cc-c_n}\geq 0.$$
Since $v(\cdot,\cc)=g(\cdot)$ satisfies
$\LL_{\cc} g(\cdot)-\TT g(\cdot)+h-\cc=0 $ in $\R^+,$
and
\begin{align*}
\LL_{c_n} v(\cdot,c_n)-\TT v(\cdot,c_n)+h-c_n=0\quad\hbox{in}~~ [0,\sw(c_n)],
\end{align*}
we have
\begin{align}\label{un}
\LL_{c_n} u^n-\TT u^n=g'-1\quad\hbox{in}~~ [0,\sw(c_n)].
\end{align}
By \eqref{vc} and \eqref{un}, $u^n$ is uniformly bounded for all $n$, and $u^n$ is bounded in $W^{1,\infty}([0,L])$ for any $L>0$ when $\sw(c_n)>L$, so it has a subsequence converges to a $u\in W^{1,\infty}(\R^+)$ weekly in $W^{1,\infty}([0,L])$ and uniformly in $C([0,L])$ for any $L>0$ which satisfies
\begin{align}\label{ug}
\LL_{\cc} u-\TT u=g'-1 \quad\hbox{in}~~ \R^+.
\end{align}
Differentiating it leads to
$\LL_{\cc} u'-\II u'=g''\leq 0$ in $\R^+$.
Applying \lemref{lem:CP1} yields $u'\leq 0$.
Since $u$ is bounded, the following limits exist
$$
u(+\infty)=\lim\limits_{x\to+\infty}u(x),\quad \limsup\limits_{x\to+\infty}u'(x)=0.
$$
So there exists a sequence $x_n\to+\infty$ such that $-\frac{1}{n}\leq u'(x_n)\leq 0$.
Noticing \eqref{g5}, by
taking $x=x_n$ in \eqref{ug} and sending $n\to \infty$, we conclude $u(+\infty)=-1/r$. But this contradicts $0\leq u^n\to u$.
\end{proof}

\subsection{Proof of \lemref{lemma:swb}}\label{proof:lemma:swb}
First, the monotonicity of $x\mapsto\swb(x,c)$ follows from \lemref{lem:N} immediately.
We now show the right-continuity property. Since $\swb(x,c)$ is non-decreasing w.r.t. $x$, we only need to prove $$\swb(x,c)\geq c^*:=\limsup\limits_{y\to x+} \swb(y,c).$$
Let $\{y_{n}\}$ be a sequence such that $y_n\to x+$ and $\swb(y_n,c)\to c^*$ as $n\to\infty$. By definition, we have $v(y_n,\swb(y_n,c))=v(y_n,c)$, so it follows from the continuity of $v$ that $v(x,c^*)=v(x,c)$ which by definition gives $ \swb(x,c)\geq c^*$.
\par
We now prove the equation \eqref{v_c=0}. Suppose $c<s:=\swb(x,c)<\cc$.
Thanks to \lemref{lem:ucL} and monotonicity, we have
\begin{align*}
0\leq \frac{v(x,s)-v(x,s+\Ds)}{\Ds}\leq \frac{v(x,s-\Ds)-v(x,s)}{\Ds}+2B\Ds=2B\Ds.
\end{align*}
for $0<\Ds<\min\{s-c,\cc-s\},$ since $v(x,\xi)=v(x,c)$ for all $\xi\in [c, s]$.
Sending $\Ds\to 0+$, we get $v_c(x,s)=0$.

\begin{appendix}

\section{Appendix}\label{sec:proofs}
\subsection{Proof of \lemref{lem:lipchitz}}
\label{proof:lem:lipchitz}
The monotonicity in $c$ and $x$ is due to $\Pi_{x,c}\subseteq \Pi_{x,c'}$ for $c'<c\leq \cc$
and $\Pi_{x,c}\subseteq \Pi_{y,c}$ for $x\leq y$.

Suppose $x\leq y$.
For any $\{(C_t,D_t)\}_{t\geq0}\in \Pi_{y,c}$, let
$D^*_t=(D_t+(y-x))\mathbf{1}_{t\geq 0}$, then $\{(C_t,D^*_t)\}_{t\geq0}\in \Pi_{x,c}$, and 
\begin{align*}
&~\E \Bigg[\int_0^\infty e^{-r t} C_t \dt-\ell \int_0^\infty e^{-r t} \d D_t \Bigg]\\
=&~\E \Bigg[\int_0^\infty e^{-r t} C_t \dt-\ell \int_0^\infty e^{-r t} \d D^*_t
+\ell(y-x) \Bigg]\\
\leq &~V(x,c)+\ell(y-x),
\end{align*}
which implies $V(y,c)\leq V(x,c)+\ell(y-x).$
When $x< 0$, we have that $\{(C_t,D_t)\}_{t\geq0}\in \Pi_{0,c}$ if and only if $\{(C_t,D_t-x)\}_{t\geq0}\in \Pi_{x,c}$, which hence implies a strategy $\{(C_t,D_t)\}_{t\geq0}$ is an optimal for $V(0,c)$ if and only if $\{(C_t,D_t-x)\}_{t\geq0}$ is optimal for $V(x,c)$. This together with the above argument leads to $V(x,c)=V(0,c)+\ell x$ for $x\leq 0$.

To show the concavity w.r.t $x$, suppose $\{(C^x_t,D^x_t)\}_{t\geq0}\in \Pi_{x,c}$
and $\{(C^y_t,D^y_t)\}_{t\geq0}\in \Pi_{y,c}$.
Then one can check $\{(C^*_t,D^*_t)\}_{t\geq0}\in \Pi_{\frac{x+y}{2},c}$ where
$C^*_t=\frac{1}{2}(C^x_t+C^y_t)$ and $D^*_t=\frac{1}{2}(D^x_t+D^y_t)$, so
\begin{align*}
&~\frac{1}{2}\E \Bigg[\int_0^\infty e^{-r t} C^x_t \dt-\ell \int_0^\infty e^{-r t} \d D^x_t \Bigg]\\
&~+\frac{1}{2}\E \Bigg[\int_0^\infty e^{-r t} C^y_t \dt-\ell \int_0^\infty e^{-r t} \d D^y_t \Bigg]\\
=&~\E \Bigg[\int_0^\infty e^{-r t} C^*_t \dt-\ell \int_0^\infty e^{-r t} \d D^*_t \Bigg]
\leq ~ V\Big(\frac{x+y}{2},c\Big).
\end{align*}
By taking supremes over all $\{(C^x_t,D^x_t)\}_{t\geq0}\in \Pi_{x,c}$ and $\{(C^y_t,D^y_t)\}_{t\geq0}\in \Pi_{y,c}$ in order, we get the desired concavity.

Since $C_t\leq \cc$ and $D_t\geq 0$ is non-decreasing for any admissible $\{(C_t,D_t)\}_{t\geq0}\in \Pi_{x,c}$, it follows
\begin{align*}
V(x,c) &\leq
\sup\limits_{\{(C_t,D_t)\}_{t\geq0}\in \Pi_{x,c}}\E \Bigg[\int_0^\infty e^{-r t} \cc \dt \Bigg]=\frac{\cc}{r},
\end{align*}
which gives a global bound for $V$ on $\R\times (-\infty,\cc]$.

To establish a lower bound for $V$ on $\R^+\times(-\infty,\cc]$, we let $D^*_t=\sum_{i=1}^{N_t} Z_i$, then $\E[D^*_t]=\la\ga t$. By integration by parts,
\begin{align*}
\E \Bigg[\int_{0}^\infty e^{-r t} \d D^*_t \Bigg]
=&~\E \Bigg[e^{-r t} D^*_t\Big|_{0-}^{\infty}+r\int_0^\infty e^{-r t} D^*_t \dt\Bigg]=\frac{\la\ga}{r}.
\end{align*}
When $x\geq 0$, one can check that $\{(\cc,D^*_t)\}_{t\geq0}\in \Pi_{x,c}$, so
\begin{align*}
V(x,c) &\geq
\E \Bigg[\int_0^\infty e^{-r t} \cc \dt-\ell\int_{0}^\infty e^{-r t} \d D^*_t \Bigg]= \frac{ \cc-\la\ell\ga}{r}.
\end{align*}
This establishes the boundedness of $V$ on $\R^+\times(-\infty,\cc]$ and
completes the proof of \lemref{lem:lipchitz}.

\subsection{Proof of \lemref{lem:CP1}}
\label{proof:lem:CP1}
We assume the first case holds, that is, $\psi_1,\;\psi_2\in W^{1,\infty}(\R^+)$ satisfy
\begin{align*}
\begin{cases}
\LL_c \psi_1-\JJ \psi_1+H(x,\psi_1)\leq\LL_c \psi_2-\JJ \psi_2+H(x,\psi_2) {\rm\quad a.e.\;in}~~ \D,\medskip\\
\psi_1\leq \psi_2{\rm\quad in}~~ \R^+\setminus\D.
\end{cases}
\end{align*}
Let $\phi(x)=e^{\frac{r}{\mu-c}x}.$ Then
\begin{align}\label{Lphi1}
\LL_c \phi(x)-\JJ \phi(x) \geq\LL_c \phi(x)-\la\phi(x)=0,\quad x\in\R^+.
\end{align}
The claim will follow if we can show, for any $\ep>0$,
\begin{align*}
\psi_1 \leq \psi_2+\ep \phi \quad \text{in}~~ \mathbb{R}^+.
\end{align*}
We now establish the above by contradiction. Assume, on the contrary,
$$
M:=\sup_{x\in \mathbb{R}^+}\big(\psi_1-\psi_2-\ep \phi\big)(x)>0
$$
for some $\ep>0$. Clearly, it implies
$
M=\sup_{x\in \mathbb{R}^+}\big(\psi_1-\psi_2-\ep \phi\big)^+(x).
$

Note that $\psi_1-\psi_2-\ep \phi$ tends to $-\infty$ as $x\to+\infty$, so it attains its maximum $M$ at some point $x_0\in \R^+$. By continuity and $\frac{r }{r+\la}M<M$, there exists a sufficiently small $\delta>0$ such that
\begin{align}\label{CP12}
\psi_1-\psi_2-\ep \phi >\frac{r }{r+\la}M \quad \text{in}~~ [x_0,x_0+\delta].
\end{align}
This implies $\psi_1> \psi_2$, so $[x_0,x_0+\delta] \subset \D$, whence
\begin{align}\label{CP_eq}
\LL_c \psi_1-\JJ \psi_1+H(x,\psi_1)\leq\LL_c \psi_2-\JJ \psi_2+H(x,\psi_2) \quad \text{a.e. in}~~ [x_0,x_0+\delta].
\end{align}
Moreover, since $H(x,y)$ is non-decreasing in $y$ and $\psi_1> \psi_2$, we have
\begin{align}\label{CP_H}
H(x,\psi_1(x))\geq H(x,\psi_2(x)), \quad x\in [x_0,x_0+\delta].
\end{align}
Since both $\LL_c$ and $\JJ$ are linear operators,
we deduce from \eqref{Lphi1}, \eqref{CP_eq}, \eqref{CP_H} and \eqref{J} that
$$
\LL_c (\psi_1-\psi_2-\ep \phi) \leq \JJ (\psi_1-\psi_2-\ep \phi)\leq
\la \sup_{y\in \mathbb{R}^+} \big(\psi_1-\psi_2-\ep \phi\big)^+(y)=
\la M \quad \text{a.e. in}~~ [x_0,x_0+\delta].
$$
It follows
$$
-(\mu-c) (\psi_1-\psi_2-\ep \phi)' \leq-(r+\la)(\psi_1-\psi_2-\ep \phi)+\la M <0 \quad \text{a.e. in}~~ [x_0,x_0+\delta],
$$
where the last inequality is due to \eqref{CP12}.
Thus, $(\psi_1-\psi_2-\ep \phi)'>0$ a.e. in $[x_0,x_0+\delta]$, contradicting that $x_0$ is a maximizer of $(\psi_1-\psi_2-\ep \phi)$. This completes the proof for the first case.

We now consider the second case:
\begin{align*}
\min\big\{\LL_c \psi_1-\JJ \psi_1+H(x,\psi_1), \; \psi_1-\eta\big\}\leq
\min\big\{\LL_c \psi_2-\JJ \psi_2+H(x,\psi_2), \; \psi_2-\eta\big\}~~{\rm a.e.\;in}~~ \R^+.
\end{align*}
Suppose
$\D=\{\psi_1> \psi_2\}$ is not empty. Then the above leads to
\begin{align*}
\LL_{c} \psi_1-\JJ \psi_1+H(x,\psi_1)\leq\LL_{c} \psi_2-\JJ \psi_2+H(x,\psi_2) {\rm\quad a.e.\;in}~~ \D.
\end{align*}
This means $\psi_1$ and $\psi_2$ satisfy the first case, so by the previous result, we have $\psi_1\leq \psi_2$ in $\R^+$. But this clearly contradicts to that $\D=\{\psi_1> \psi_2\}$ is not empty. This completes the proof of \lemref{lem:CP1}.

\subsection{Proof of \lemref{lem:g}}
\label{proof:lem:g}
The existence of a $C^1$-smooth bounded solution to the equation \eqref{g_eq} can be proved by constructing the approximation problem in bounded interval and applying the Leray-Schauder fixed point theorem (see \cite{Ev16} Theorem 4 on page 539), we leave it to the interested readers.
The uniqueness of the solution is a consequence of \lemref{lem:CP1}.

Next, we prove \eqref{g_b}.
Let $\phi=\frac{\cc-\la\ell\ga}{r}$ and $\Phi=\frac{\cc}{r}$.
Since
$$0\leq h(x)=\la\ell \int_x^\infty (1-F(y))\dy \leq \la\ell \int_0^\infty (1-F(y))\dy=\la\ell\ga,$$
it follows
\begin{align*}
\LL_{\cc} \phi-\TT \phi+h-\cc=h-\la\ell\ga\leq 0 \leq h=\LL_{\cc} \Phi-\TT \Phi+h-\cc \quad {\rm in}~~ \R^+.
\end{align*}
Applying \lemref{lem:CP1}, we obtain \eqref{g_b}.

We come to prove \eqref{g'}. Differentiating \eqref{g_eq} we have
\begin{align}\label{g'_eq}
\LL_{\cc} \big(g'\big)-\II \big(g'\big)=\la\ell (1-F) \quad {\rm in}~~ \R^+,
\end{align}
Let $\psi=0$ and $\Psi=\ell$, then
\begin{align*}
\LL_{\cc} \psi-\II \psi=0 \leq \la\ell (1-F) \leq r \ell+\la\ell (1-F)
=\LL_{\cc} \Psi-\II \Psi \quad {\rm in}~~ \R^+.
\end{align*}
Applying \lemref{lem:CP1} to \eqref{g'_eq}, we have \eqref{g'}.

We now prove \eqref{g''}.
Differentiating \eqref{g'_eq} yields
\begin{align*}
\LL_{\cc} \big(g''\big)-\II \big(g''\big)=\la ( g'(0)-\ell )p \quad {\rm a.e.\; in}~~ \R^+.
\end{align*}
Since $g'$, $g''$ and $p$ are bounded in $\R^+$, it follows $g''\in W^{1,\infty}(\R^+)$. Also, since $g'\leq \ell$ and $\la p\geq 0$, we have
\begin{align*}
\LL_{\cc} \big(g''\big)-\II \big(g''\big) \leq 0 \quad {\rm a.e.\; in}~~ \R^+.
\end{align*}
Applying \lemref{lem:CP1}, we obtain \eqref{g''}.

Finally, we prove \eqref{g4} and \eqref{g5}. From \eqref{g_b}-\eqref{g''} we know the limits $g(+\infty)$ and $g'(+\infty)$ exist and \eqref{g5} holds. Let $x\to+\infty$ in \eqref{g_eq} and using the dominated convergence theorem we get \eqref{g4} and complete the proof of \lemref{lem:g}.

\subsection{Proof of \thmref{theorem:boundaryvalue}}
\label{proof:theorem:boundaryvalue}

The proof consists of two steps.
\paragraph{Step 1 (Upper bound: $g(x)\geq V(x,\cc)$).}

Let $x\in\R^+$ and $\{(C_t,D_t)\}_{t\geq0}\in \Pi_{x,\cc}$ be any admissible control.
In the boundary case, we have $C_t\equiv \cc$.
For any subset $\mathbb{K}$ of $\R$ with zero Lebesgue measure,
one can show through a change of measure argument that $\int_{0}^{T}\P(X_t\in \mathbb{K})\dt=0$ for any $T>0$.
Applying the general It\^o's formula (see \cite{Pr90} Theorem 33 on page 81) to $e^{-r t}g(X_t)$ and then taking expectation yields
\begin{align}
g(x)
=&~
\E \Bigg[e^{-c T}g(X_{T})\Bigg]-\E \Bigg[\sum_{0\leq t\leq T}e^{-r t}\Big[g (X_t)-g(X_{t-})\Big]\Bigg]\nn\\
&~-\E \Bigg[\int_0^{T}e^{-r t}\Big[(\mu-\cc) g' -r g \Big](X_{t-}) \dt\Bigg]-\E \Bigg[\int_0^{T}e^{-r t} g'(X_{t-})\d D^c_t\Bigg].\label{gE}
\end{align}
Since $\Delta D_t=X_t-X_{t-}+Z_{N_t} \Delta N_t$ and $X_{t}$, $\Delta D_t\geq 0$,
we get
$$\Delta D_t\geq \Delta' D_t:=(Z_{N_t} \Delta N_t-X_{t-})^+=(Z_{N_t} -X_{t-})^+ \Delta N_t,$$
where the second equality holds since $\Delta N_t \in \{0,1\}$ and $X_{t-}\geq 0$.
As a consequence, we have
$$X_t-(\Delta D_t-\Delta' D_t)=X_{t-} - Z_{N_t} \Delta N_t + \Delta' D_t
=X_{t-}(1-\Delta N_t)+(X_{t-}-Z_{N_t})^+\Delta N_t.$$
The second expectation on the right-hand side of \eqref{gE} can then be decomposed as
\begin{align*}
&\E \Bigg[\sum_{0\leq t\leq T}e^{-r t}\Big[g(X_t)
-g\Big(X_t-(\Delta D_t-\Delta' D_t)\Big)\Big]\Bigg]\\
&+\E \Bigg[\sum_{0\leq t\leq T}e^{-r t}\Big[g\Big(X_{t-}(1-\Delta N_t)+(X_{t-}-Z_{N_t})^+\Delta N_t\Big)-g(X_{t-})\Big]\Bigg].
\end{align*}
Since $\Delta N_t \in \{0,1\}$, the second term in above can be rewritten as
\begin{align*}
&~\E \Bigg[\sum_{0\leq t\leq T}e^{-r t}\Big[g\Big(X_{t-}(1-\Delta N_t)+(X_{t-}-Z_{N_t})^+\Delta N_t\Big)- g(X_{t-})\Big]\Bigg]\\[2mm]
=&~\E \Bigg[\sum_{0\leq t\leq T}e^{-r t}\Big[g\Big((X_{t-}-Z_{N_t})^+\Big)- g(X_{t-})\Big] \Delta N_t\Bigg]\\
=&~\E \Bigg[\int_0^{T} e^{-r t}\Big[g\Big((X_{t-}-Z_{N_t})^+\Big)-g(X_{t-})\Big] \la \dt\Bigg]\\
=&~\E \Bigg[\int_0^{T} e^{-r t}\Big(\E_{t,X_{t-}}\Big[g\Big((X_{t-}-Z_{N_t})^+\Big)\Big]-g(X_{t-})\Big) \la \dt\Bigg]\\
=&~\E \Bigg[\int_0^{T}e^{-r t}\Big[\TT g(X_{t-} ) - \la g(X_{t-})\Big] \dt\Bigg],
\end{align*}
where the second equality follows from the martingale property of the compensated Poisson process $\{N_t- \la t\}_{t\geq 0}$ (see \cite{Sh04}), and the third equality follows from the law of iterated expectations. Substituting these expressions into \eqref{gE} yields
\begin{align}\nonumber
g(x)
=&~
\E \Bigg[e^{-r T}g(X_{T})\Bigg]
-\E \Bigg[\sum_{0\leq t\leq T}e^{-r t}\left[g(X_t)-g\Big(X_t-(\Delta D_t-\Delta' D_t)\Big)\right]\Bigg]\\\label{gE1}
& + \E \Bigg[\int_0^{T}e^{-r t} \Big(\LL_{\cc} g - \TT g\Big)(X_{t-},t) \dt\Bigg]
-\E \Bigg[\int_0^{T}e^{-r t} g'(X_{t-})\d D^c_t\Bigg].
\end{align}
Observe that
\begin{align*}
\ell\E \Bigg[\sum_{0\leq t\leq T}e^{-r t}\Delta' D_t \Bigg]
=&~\ell \E \Bigg[\sum_{0\leq t\leq T}e^{-r t} (Z_{N_t}-X_{t-})^+ \Delta N_t\Bigg]\\
=&~\ell \E \Bigg[\int_0^T e^{-r t} (Z_{N_t}-X_{t-})^+ \la \dt\Bigg]\\
=&~\ell \la \E \Bigg[\int_0^T e^{-r t} \E_{t,X_{t-}}\Big[(Z_{N_t}-X_{t-})^+\Big] \dt\Bigg]= \E\Bigg[\int_0^T e^{-r t} h(X_{t-})\dt\Bigg].
\end{align*}
Then \eqref{gE1} becomes 
\begin{align}\nonumber
g(x)= &~\E \Bigg[e^{-r T}g(X_{T})\Bigg]-\E \Bigg[\sum_{0\leq t\leq T}e^{-r t}\left[g(X_t)-g\Big(X_t-(\Delta D_t-\Delta' D_t)\Big)\right]\Bigg] \\\nonumber
&~+ \E \Bigg[\int_0^{T}e^{-r t} \Big(\LL_{\cc} g - \TT g + h - \cc \Big)(X_{t-}) \dt\Bigg]-\E \Bigg[\int_0^{T}e^{-r t} g'(X_{t-})\d D^c_t\Bigg]\\\label{gE2}
&~- \ell \E \Bigg[\sum_{0\leq t\leq T}e^{-r t}\Delta' D_t \Bigg]
+ \E\Bigg[\int_0^T e^{-r t} \cc \dt\Bigg].
\end{align}
The definition of $g$ implies that the third expectation vanishes. Combining $0\leq g' \leq \ell$ and $\Delta D_t - \Delta' D_t \geq 0$, we obtain
\begin{align*}
g(x)\geq \E \Bigg[e^{-r T}g(X_{T})\Bigg] + \E\Bigg[\int_0^T e^{-r t} \cc \dt -\int_0^T e^{-r t} \ell \d D_t\Bigg]
\end{align*}
Since $g$ is bounded in $\R^+$, by the dominated convergence theorem and monotone convergence theorem, sending $T \to +\infty$ yields
\begin{align*}
g(x)
\geq \E\Bigg[\int_0^\infty e^{-r t} \cc \dt -\int_0^\infty e^{-r t} \ell \d D_t\Bigg].
\end{align*}
Therefore, we conclude $g(x)\geq V(x,\cc)$.

\paragraph{Step 2 (Lower bound and optimality: $g(x)\leq V(x,\cc)$).}
We now prove the opposite inequality and verify that the candidate capital injection strategy \begin{align*}
\overline{D}_t=\sup\limits_{0\leq s\leq t}\Big(\sum_{i=1}^{N_s} Z_i-x-(\mu-\cc)s\Big)^+,~~t\geq 0.
\end{align*}
that proposed in \thmref{theorem:boundaryvalue} is indeed optimal.
Since $\overline{D}_t$ is non-decreasing and the surplus
$$\overline{X}_t=x + (\mu-\cc)t - \sum_{i=1}^{N_t} Z_i + \overline{D}_t \geq 0, ~~ t\geq 0, $$
we see $\{(\cc, \overline{D}_t)\}_{t\geq0}\in \Pi_{x,c}$.
We next show
\begin{align}\label{DD}
\Delta \overline{D}_t = \Delta' \overline{D}_t,~~ t \geq 0.
\end{align}
Indeed, there are two cases.
\begin{itemize}
\item $\Delta \overline{D}_t > 0$.
In this case, we have
$\overline{D}_t = \sum_{i=1}^{N_t} Z_i - x - (\mu - \bar{c})t$, so
\begin{align*}
\Delta \overline{D}_t =&~ (\overline{D}_t - \overline{D}_{t-})^+ \\
=&~ \Bigg[\Big(\sum_{i=1}^{N_t} Z_i-x-(\mu-\cc) t \Big)- \Big(\overline{X}_{t-} - x - (\mu-\cc) t + \sum_{i=1}^{N_{t-}} Z_i\Big) \Bigg]^+ \\
=&~ (Z_{N_t}\Delta N_t-\overline{X}_{t-})^+ = \Delta' \overline{D}_t.
\end{align*}

\item $\Delta \overline{D}_t =0$. In this case, we have $\overline{X}_t - \overline{X}_{t-} = -Z_{N_t} \Delta N_t$, so
$$
\Delta' \overline{D}_t = (Z_{N_t} \Delta N_t - \overline{X}_{t-})^+ = (-\overline{X}_t)^+ = 0 = \Delta \overline{D}_t.
$$
\end{itemize}
Thus, we proved \eqref{DD}.
Moreover, it is straightforward to verify that $\overline{D}_t^c \equiv 0$.
Taking these into \eqref{gE2} yields
\[
g(x) = \E \left[ e^{-rT} g(\overline{X}_T) \right] + \E \left[ \int_0^T e^{-rt} \bar{c} \dt - \int_0^T e^{-rt} \ell \d \overline{D}_t \right].
\]
By sending $T \to \infty$, we obtain
\[
g(x) = \E \left[ \int_0^\infty e^{-rt} \bar{c} \dt - \int_0^\infty e^{-rt} \ell \, \d \overline{D}_t \right].
\]
Now we conclude $\{(\cc, \overline{D}_t)\}_{t\geq0}$ is an optimal strategy and $g$ is the value function for the boundary case. This completes the proof of \thmref{theorem:boundaryvalue}.

\subsection{Proof of \lemref{lem:CP00}}
\label{proof:lem:CP00}
Suppose $v,\;w\in \A_{\infty}$ are two bounded strong solutions to the VI \eqref{v_pb00},
and $v>w$ at some point in $\Q^{+}_{\infty}$.
Then by continuity, there exist $\ep>0$ and $\uc<\cc$ such that
$$M:=\sup\limits_{(x,c)\in \R^+\times [\uc,\cc]}\big((1+\ep)v-w-\ep \phi\big)(x,c)>0,$$
where $\phi(x)=x+a$ with $a>(\cc+\mu-\uc)/r$.
Clearly,
$$M=\sup\limits_{(x,c)\in \R^+\times [\uc,\cc]}\big((1+\ep)v-w-\ep \phi\big)^+(x,c).$$

Note that $(1+\ep)v-w-\ep \phi$ is continuous and tends to $-\infty$ as $x\to+\infty$, and
$\big((1+\ep)v-w-\ep \phi\big)(x,\cc)=\ep (g(x)- \phi(x))\leq \ep (\cc/r- a)<0$ for $x\in\R^+$,
so $(1+\ep)v-w-\ep \phi$ attains its maximum value at some point $(x_0,c_0)\in \R^+\times[\uc,\cc)$. It hence follows
\begin{align}\label{px<=0}
\p_x \big((1+\ep)v-w-\ep \phi\big)(x_0,c_0) \leq 0.
\end{align}
This together with $\mu\geq c$ and
$$\TT \big((1+\ep)v-w-\ep \phi\big)(x_0,c_0)\leq \la \big((1+\ep)v-w-\ep \phi\big)^+(x_0,c_0)=\la M,$$
implies
\begin{multline}\label{estimate3}
\qquad\LL_{c_0} \big((1+\ep)v-w-\ep \phi\big)(x_0,c_0)-\TT \big((1+\ep)v-w-\ep \phi\big)(x_0,c_0)\\
\geq (r+\la)\big((1+\ep)v-w-\ep \phi\big)(x_0,c_0)-\la M=rM >0.\qquad
\end{multline}
On the other hand, we may assume $c_0$ satisfies
\begin{align*}
\big((1+\ep)v-w-\ep \phi\big)(x_0,c)<M=\big((1+\ep)v-w-\ep \phi\big)(x_0,c_0), \quad c_0<c\leq \cc.
\end{align*}
Since $w+\ep \phi$ is non-increasing w.r.t. $c$, it follows
\begin{align*}
v(x_0,c_0)> v(x_0,c),\quad c_0<c\leq \cc.
\end{align*}
Then by \defref{def:solution}, we have
$$\LL_{c_0} v(x_0,c_0)-\TT v(x_0,c_0)+h(x_0)-c_0=0\;\hbox{ or }\; v_x(x_0,c_0)=\ell.$$
Notice
\begin{align*}
\LL_{c_0} \phi(x_0)-\TT \phi(x_0) &\geq\LL_{c_0} \phi(x_0)-\la\phi(x_0)\nn\\
&=rx_0+ra-\mu+c_0\geq ra-\mu+c_0> \cc.
\end{align*}
If $\LL_{c_0} v(x_0,c_0)-\TT v(x_0,c_0)+h(x_0)-c_0=0$,
using $\LL_c w-\TT w+h-c\geq 0$ and above, we get
$$
\LL_{c_0} \big((1+\ep)v-w-\ep \phi\big)(x_0,c_0)-\TT \big((1+\ep)v-w-\ep \phi\big)(x_0,c_0)<-\ep (h(x_0)-c_0+\cc)< 0,
$$
which contradicts \eqref{estimate3}.
If $v_x(x_0,c_0)=\ell $, combining $w_x(x_0,c_0) \leq \ell$ and $\phi_x(x_0,c_0)=1$, we have
$\p_x\big((1+\ep)v-w-\ep \phi\big)(x_0,c_0)\geq \ep (\ell-1)>0$, which contradicts \eqref{px<=0}. This completes the proof of \lemref{lem:CP00}.

\end{appendix}

 \bibliographystyle{plainnat}

\renewcommand{\baselinestretch}{1.2}

\end{document}